\documentclass[a4paper,11pt]{article}

\usepackage[notext]{kpfonts}
\usepackage{baskervald}

\usepackage{amsfonts,amssymb,amsmath,amsthm,abstract,color}
\usepackage[mathscr]{euscript}
\usepackage[ps,all,arc,rotate]{xy}
\usepackage[lmargin=1in,rmargin=1in,tmargin=1in,bmargin=1in]{geometry}
\usepackage{fancyhdr}
\usepackage{dsfont}
\usepackage{pb-diagram}
\usepackage{enumitem}
\usepackage[backref=page]{hyperref}
\usepackage[usenames,dvipsnames]{xcolor}
\usepackage{tikz-cd}
\usepackage{extpfeil}
\usepackage{float}
\hypersetup{colorlinks=true,citecolor=NavyBlue,linkcolor=Brown,urlcolor=Orange}

\usepackage{stmaryrd}
\usepackage[capitalise]{cleveref}
\numberwithin{equation}{section}


\usepackage{titlesec}

\titlespacing*{\section}
{0pt}{3ex plus 1ex minus .2ex}{2.5ex plus .2ex}
\titlespacing*{\subsection}
{0pt}{2ex plus 1ex minus .2ex}{1.5ex plus .2ex}
\titlespacing*{\subsubsection}
{0pt}{2ex plus 1ex minus .2ex}{1.5ex}

\newcommand{\sectionpoint}{\if \@empty\titlesec \else}              

\newcommand*{\justifyheading}{\raggedright}
\titleformat{\section}{\normalfont \Large \bfseries \justifyheading}{\thesection.}{0.5em}{}
\titleformat{\subsection}[runin]{\normalfont \large \bfseries \justifyheading}{\thesubsection.}{0.5em}{}[.]
\titleformat{\subsubsection}[runin]{\normalfont \normalsize \bfseries \justifyheading}{\thesubsubsection.}{0.5em}{}[  \protect{\rule[3pt]{10pt}{0.5pt}}]

\theoremstyle{plain}
\newtheorem*{theorem*}{Theorem}
\newtheorem{theorem}{Theorem}[section]
\newtheorem{prop}[theorem]{Proposition}

\newtheorem{lemma}[theorem]{Lemma}
\newtheorem{cor}[theorem]{Corollary}

\theoremstyle{definition}

\theoremstyle{remark}
\newtheorem{rmk}[theorem]{Remark}

\newtheorem{ex}[theorem]{Example}
\newtheorem{question}[theorem]{Question}

\renewcommand{\tilde}{\widetilde}

\newcommand{\CC}{\mathbb{C}}
\newcommand{\RR}{\mathbb{R}}
\DeclareMathOperator{\F2}{\mathbb{F}_2}

\newcommand{\PP}{\mathbb{P}}
\newcommand{\QQ}{\mathbb{Q}}

\newcommand{\ZZ}{\mathbb{Z}}

\newcommand{\m}{\mathfrak{m}}
\newcommand{\p}{\mathfrak{p}}

\newcommand{\z}{\mathfrak{z}}

\renewcommand{\1}{\mathds{1}}

\DeclareMathOperator{\End}{End}

\DeclareMathOperator{\GL}{GL}

\DeclareMathOperator{\Hom}{Hom}
\DeclareMathOperator{\Hilb}{Hilb}
\DeclareMathOperator{\id}{id}

\DeclareMathOperator{\Mon}{Mon}

\DeclareMathOperator{\pt}{pt}
\DeclareMathOperator{\rk}{rk}

\DeclareMathOperator{\Sym}{Sym}

\newcommand{\Aut}{\mathrm{Aut}}

\newcommand{\Bl}{\mathrm{Bl}}
\renewcommand{\div}{\mathrm{div}}
\newcommand{\Pic}{\mathsf{Pic}}

\newcommand{\Hdg}{\mathrm{Hdg}}

\newcommand{\sgn}{\mathsf{sgn}}

\newcommand{\tors}{\mathrm{tors}}

\renewcommand{\Im}{\mathrm{Im}}

\DeclareMathOperator{\K3}{\mathrm{K3}}

\begin{document}

\title{\textbf{On maximality of involutions of hyper-K\"ahler manifolds and punctual Hilbert schemes of surfaces}}

\author{Simone Billi, Lie Fu, Annalisa Grossi, Viatcheslav Kharlamov}
\date{ }
\maketitle

\begin{abstract}
Given a holomorphic or anti-holomorphic involution on a complex variety, the Smith inequality says that the total $\F2$-Betti number of the fixed locus is no greater than the total $\F2$-Betti number of the ambient variety. The involution is called maximal when the equality is achieved. In this paper, we investigate maximality of involutions of compact hyper-K\"ahler manifolds and of Hilbert schemes of points on surfaces. We obtain both positive and negative results. 

On one hand, given a smooth projective surface $S$ with $H^1(S, \F2)=0$ equipped with a holomorphic (resp.~anti-holomorphic) involution $\sigma$, we establish the following necessary and sufficient condition for the maximality of the induced involution on the $n$th Hilbert scheme of points: the induced involution is maximal if and only if $\sigma$ is a maximal involution of $S$ and it acts on $H^2(S, \ZZ)$ trivially (resp.~as $-\id$). This generalizes and completes previous partial results of Fu and Kharlamov--R\u asdeaconu.

On the other hand, we show that for $n\geq 2$, a hyper-Kähler manifold of K3\(^{[n]}\)-deformation type admits neither maximal anti-holomorphic involutions (i.e.~real structures), nor maximal holomorphic (symplectic or anti-symplectic) involutions. In other words, such hyper-K\"ahler manifolds do not admit maximal (AAB), (ABA), (BAA) or (BBB) brane involutions in the sense of Kapustin--Witten.
\end{abstract}

\medskip

\renewcommand{\thefootnote}{}
\null\footnotetext{L.F.\ is supported by the University of Strasbourg Institute for Advanced Study (USIAS), by the Agence Nationale de la Recherche (ANR) under projects ANR-20-CE40-0023 and ANR-24-CE40-4098, and by the International Emerging Actions (IEA) project of CNRS.
S.B. and A.G. were partially supported by the European Union - NextGenerationEU under the National Recovery and Resilience Plan (PNRR) - Mission 4 Education and research - Component 2 From research to business - Investment 1.1 Notice Prin 2022 - DD N. 104 del 2/2/2022, from title \lq\lq Symplectic varieties: their interplay with Fano manifolds and derived categories\rq\rq, proposal code 2022PEKYBJ – CUP J53D23003840006.
S.B. and A.G. are members of the INdAM group GNSAGA, and was partially supported by GNSAGA.

}

\renewcommand{\thefootnote}{\arabic{footnote}} 

\small{\tableofcontents}
\newpage

\section{Introduction}

\subsection{Smith inequality}

For a topological space equipped with an involution satisfying mild conditions \footnote{For example,
a topological space 
admitting the structure of a CW-complex that is respected by the involution, which is always the case for smooth involutions on differentiable manifolds.}, the Smith theory relates the topology of the fixed locus and that of the ambient space. In particular, we have the following fundamental inequality relating their total $\F2$-Betti numbers (see for example \cite{Bredon}):

\begin{theorem}[Smith inequality]
	\label{thm:SmithThom}
	Let $X$ be a topological space and $\sigma$ an involution of $X$. Assume that $X$ has the structure of a finite simplicial complex that is respected by $\sigma$.
	Let $X^{\sigma}$ be the fixed locus. We have the following inequality for the total $\F2$-Betti numbers 
	\begin{equation}
		\label{eqn:SmithThomInequality}
		b_*(X^{\sigma}, \F2)\leq b_*(X, \F2).
	\end{equation}
\end{theorem}
Recall that for a topological space $W$, its total $\F2$-Betti number is defined as $b_*(W, \F2):=\sum_{i} b_i(W, \F2)$ with $b_i(W, \F2):=\dim_{\F2} H^i(W, \F2)$.

When the equality holds in \Cref{thm:SmithThom}, we say that the pair $(X, \sigma)$ is \textit{maximal}. With a slight abuse of terminology, we also say that the involution $\sigma$ is maximal (when $X$ is clear from the context), or $X$ is maximal (when the involution is the natural one, e.g.~the real structure when $X$ is defined over $\RR$).

\subsection{Motivation from real geometry}
A \textit{real structure} on a complex manifold $X$ is an anti-holomorphic involution, that is, a diffeomorphism
\[\sigma \colon X\to X\]
satisfying $\sigma^2=\id_X$ and $\sigma^*I=-I$, where $I$ denotes the complex structure on $X$. A \textit{real variety} (or $\mathbb{R}$-variety) refers to a pair $(X, \sigma)$ consisting of a complex manifold $X$ and a real structure $\sigma$ on it.

The \textit{real locus} of $(X, \sigma)$, denoted by $X(\RR)$, is defined to be the fixed locus of the involution $\sigma$. When $X(\RR)\neq \emptyset$,  it is a differentiable submanifold of $X$, and its real dimension is equal to the complex dimension of $X$.

By \Cref{thm:SmithThom}, for $(X,\sigma)$ a real variety, we have the following inequality for the total $\F2$-Betti numbers 
\begin{equation}
	\label{eqn:SmithThomInequality-real}
	b_*(X(\RR), \F2)\leq b_*(X, \F2).
\end{equation}
In the case where $X$ is a Riemann surface, \eqref{eqn:SmithThomInequality-real} says that the real locus can have at most $g+1$ connected components (which are circles) -- a famous classical result of Harnack \cite{Harnack} and Klein \cite{Klein-1876}.

When the equality in \eqref{eqn:SmithThomInequality-real} holds, we call $(X, \sigma)$ a \textit{maximal} real variety (or \textit{M-variety}).
Since the solution of the concrete questions on plane sextic curves and space quartic surfaces raised by D.~Hilbert in his 16-th problem, achieved in the 1970's (\cite{Gudkov}, \cite{Kharlamov-K3-76}), maximal real varieties have entered, and remain, in the center of attention in real algebraic geometry.
Many remarkable properties of these varieties were discovered, like Rokhlin's congruence theorem for even-dimensional maximal smooth projective real varieties: $\chi(X(\RR))\equiv \operatorname{sgn}(X)  \mod 16$; see for example \cite[Theorem 3.4.2]{MangolteBook}.
However, constructing examples of maximal real varieties in dimension $>2$ still remains challenging; see \cite[\S 3]{BrugalleSchaffhauser} and \cite{FuMaximalReal} for a recent summary and see \cite{MangolteBook} for the case of curves and surfaces.
Even for hypersurfaces in projective spaces, despite the powerful Viro's method of patchworking (\cite{Viro-Conf}, \cite{Viro-Gluing}, \cite{ItenbergViro}), it is still unknown whether maximal hypersurfaces exist for any degree in any dimension.
Apart from hypersurfaces and complete intersections in projective spaces, another important class of varieties are various moduli spaces. Remarkably, they
all have a natural real structure and their possible maximality is an important source of applications.
That is the reason why in the last decades maximality of moduli spaces became a subject of research. 
Let us mention here the recent contributions by  Etingof--Henriques--Kamnitzer--Rains \cite{Eingof-Henriques-Kamnitzer-Rains}, Brugall\'e--Schaffhauser \cite{BrugalleSchaffhauser} and Fu \cite{FuMaximalReal}. The present paper is a natural continuation of \cite{FuMaximalReal}, with a more focused investigation on Hilbert schemes of points on surfaces and hyper-K\"ahler manifolds.

As is shown in the seminal works of G\"ottsche \cite{Goettsche} and Nakajima \cite{Nakajima}, Hilbert schemes\footnote{It should rather be called Douady space in the complex analytic category.} of points on surfaces play distinctive roles among moduli spaces, in that this construction provides strong relations between the geometry of the surface and the geometry of an infinite series of varieties of arbitrarily high dimensions. Our initial motivation stems from the desire to determine when the Hilbert schemes of a real surface are maximal. More precisely, let us formulate the question in a boarder context as follows. Recall that the Hilbert scheme of points on a smooth surface is always smooth \cite{Fogarty}, and a holomorphic or anti-holomorphic involution on the surface naturally induces an involution on the Hilbert scheme of points by base-change (cf.~\cite[5.3.1]{CattaneoFu} for the anti-holomorphic case).
\begin{question}
\label{question:Hilb}
    Given a smooth complex surface $S$ equipped with a holomorphic or anti-holomorphic involution, when is the naturally induced involution on the Hilbert scheme of $n$ points $S^{[n]}$ maximal?
\end{question}
Partial results towards this question have been recently obtained in Fu \cite{FuMaximalReal} and in Kharlamov--R\u asdeaconu \cite{Kharlamov-Rasdeaconu-HilbertSquare}. In this paper, we give a \textit{complete} answer to it for projective surfaces satisfying $H^1(S, \F2)=0$; see \Cref{subsec:Intro-HilbertSchemes}.


\subsection{Motivation from hyper-K\"ahler geometry: brane involutions}
Given a K\"ahler manifold $X$ with complex structure $I$ and the associated symplectic form $\omega_I$, following Kapustin--Witten \cite{Kapustin-Witten_GeometricLanglands}, an \textit{$A$-brane} refers to a Lagrangian submanifold with respect to $\omega_I$, and a \textit{$B$-brane} stands for a complex submanifold with respect to $I$.

For a hyper-K\"ahler manifold $(X, g)$ with three complex structures $I, J, K$ satisfying the quaternion relations, let $\omega_I, \omega_J, \omega_K$ be the associated K\"ahler forms. It is extremely interesting to study submanifolds of $X$ that are Lagrangian or complex with respect to \textit{each} of the three K\"ahler structures. Such a submanifold is called an (AAB), (ABA), (BAA) or (BBB) brane, depending on its compatibility types with respect to the three K\"ahler structures. As is pointed out in \cite{BaragliaSchaposnik} (see also \cite[\S 2.3]{FrancoJardimMenet}), a natural and rich source of such branes in hyper-K\"ahler manifolds is provided by the fixed loci of so-called \textit{brane involutions} --- those are involutions preserving the hyper-K\"ahler metric such that each of the three complex structures is either preserved (B-type)  or reversed (A-type).

In concrete terms, on a hyper-K\"ahler manifold $X$, a real structure, or equivalently, an anti-holomorphic involution $\sigma$, gives rise to an (ABA) or (AAB) brane involution\footnote{The distinction between (ABA) and (AAB) depends on the action of $\sigma$ on 
a choice of a (unique up to scalar) holomorphic symplectic form $\eta$: it is called (ABA) if  $\sigma^*(\eta)=\overline{\eta}$, and called (AAB) if $\sigma^*(\eta)=-\overline{\eta}$.}, and its fixed locus, which is precisely the real locus $X(\RR)$, represents an (ABA) or (AAB) brane in $X$. 
Up to hyper-K\"ahler rotation of the complex structure, an anti-holomorphic involution can become a holomorphic anti-symplectic involution, i.e.~a (BAA)-brane involution, and the fixed locus, which is a holomorphic Lagrangian submanifold, is a (BAA) brane. Similarly, a holomorphic symplectic involution on a hyper-K\"ahler manifold is referred to as a (BBB)-involution and its fixed locus represents a (BBB)-brane, also known as trianalytic submanifold introduced by Verbitsky \cite{Verbitsky-Trianalytic}. 

Examples of non-compact hyper-K\"ahler manifolds with maximal (ABA) or (AAB) brane involutions are constructed via moduli spaces of Higgs bundles by Fu in \cite[Theorem 6.3]{FuMaximalReal}. Another example of maximal (ABA) or (AAB) brane involution in a non-compact hyper-K\"ahler manifold is provided by the natural real structure of the cotangent bundle of a maximal $\mathbb{R}$-variety. 

However, the situation is more intriguing for \textit{compact} hyper-K\"ahler manifolds. We refer to \cite{Beauville} and \cite{HuyInventiones} for generalities of such manifolds; see also \Cref{sec:HK-LLV} for a quick summary of their special properties. Let us just mention here the most studied examples of compact hyper-K\"ahler manifolds: deformations of Hilbert schemes of points on K3 surfaces; such hyper-K\"ahler manifolds are called of \textit{$\K3^{[n]}$-type}.
On the one hand, in (complex) dimension 2, maximal real K3 surfaces and abelian surfaces exist and have been thoroughly studied; see \cite{Kharlamov-K3-76}, \cite{nikulin1979integral}  
\cite[Chapters IV, VIII]{Silhol-Surface-LNM}. On the other hand, Kharlamov and R{\u{a}}sdeaconu \cite{Kharlamov-Rasdeaconu-HilbertSquare} made the surprising discovery that the Hilbert square of maximal real K3 surfaces (see also \cite{Kharlamov-Radsdeaconu-Deficiency} for related results) and Fano varieties of lines of maximal real cubic fourfolds are \textit{never} maximal. These works lead us to the second motivation of this paper, summarized in the following question, raised by Fu in \cite{FuMaximalReal}: 
\begin{question}
    Can compact hyper-K\"ahler manifolds of dimension $\geq 4$ admit maximal (ABA), (AAB), (BAA) and (BBB) brane involutions?
\end{question}
In what follows, we give a decisive but rather unexpected answer to it for hyper-K\"ahler manifolds of $\K3^{[n]}$-type.

\subsection{Main results I: non-existence of maximal brane involutions on $\K3^{[n]}$-type manifolds}
Although there are K3 surfaces and abelian surfaces admitting maximal brane involutions, our first main result proves the non-existence of maximal (ABA), (AAB) or (BAA) brane involutions in the compact hyper-K\"ahler manifolds of $\K3^{[n]}$-type.

\begin{theorem}[Absence of maximal (BAA)-brane involutions]
	\label{thm:main:NonMaxHKnOdd}
	Let $n\geq 2$ be an integer. Let $X$ be a hyper-K\"ahler manifold of $\K3^{[n]}$-type.
	Then $X$ does not admit maximal holomorphic anti-symplectic involutions.
\end{theorem}

By hyper-K\"ahler rotation, we immediately get the following application in real algebraic geometry, answering the aforementioned question raised by the second author in \cite{FuMaximalReal} for these most studied deformation families of compact hyper-K\"ahler manifolds. 
\begin{cor}[Absence of maximal (ABA)/(AAB)-brane involutions]
	\label{cor:main:NoMaxRealStructureHKnOdd}
	Let $n\geq 2$ be an integer. There is no maximal real structure on a hyper-K\"ahler manifold of $\K3^{[n]}$-type.
\end{cor}

We also prove the parallel result for (BBB)-brane involutions.
\begin{theorem}[Absence of maximal (BBB)-brane involutions]
\label{thm:main:NonMaxHKnSymp}
	Let $X$ be a hyper-K\"ahler manifold of $\K3^{[n]}$-type.
	Then $X$ does not admit non-trivial maximal holomorphic symplectic involutions. 
\end{theorem}

\subsection{Main results II: criteria for maximality of natural involutions on  Hilbert schemes}
\label{subsec:Intro-HilbertSchemes}
Before giving our second main result, we extract here the following special case of \Cref{thm:main:NonMaxHKnOdd}, \Cref{cor:main:NoMaxRealStructureHKnOdd} and \Cref{thm:main:NonMaxHKnSymp}.
\begin{cor}[=\Cref{cor:NaturalHilbertK3}]
	\label{cor:main:NaturalK3}
	Let  $\sigma$ be a (non-trivial) holomorphic or anti-holomorphic involution on a K3 surface $S$. Then for any $n\geq 2$, the induced involution on its Hilbert scheme of $n$ points $S^{[n]}$ is not maximal.    
\end{cor}

Going beyond hyper-K\"ahler geometry, we regard \Cref{cor:main:NaturalK3} as an example towards the investigation of the more general \Cref{question:Hilb} whether the operations $-^{[n]}$ of taking Hilbert powers on a surface preserve the maximality. For projective surfaces $S$ with $H^1(S, \F2)=0$, our results below provide a complete answer to \Cref{question:Hilb} by giving a sufficient and necessary condition, solely in terms of the involution on $S$.

\begin{theorem}[=\Cref{thm:NaturalAntiHoloInv}]
	\label{thm:main:NaturalAntiHoloInv}
	Let $n\geq 2$. Let $S$ be a smooth projective $\RR$-surface. Assume that $H^1(S, \F2)=0$.
	Then the punctual Hilbert scheme $S^{[n]}$, equipped with the natural real structure, is maximal if and only if $S$ is maximal and with connected real locus.
\end{theorem}
\Cref{thm:main:NaturalAntiHoloInv} establishes the converse of \cite[Theorem 8.1]{FuMaximalReal} and generalizes \cite[Theorem 1.2]{Kharlamov-Rasdeaconu-HilbertSquare} to all $n\ge 2$.
In a similar fashion, the following theorem gives a clean characterization for the maximality of natural \textit{holomorphic} involution on punctual Hilbert schemes of surfaces. In \Cref{subsec:Examples} these theorems are applied to obtain examples of maximal involutions and non-maximal involutions, generalizing \cite[Corollaries 1.3 and 1.4]{Kharlamov-Rasdeaconu-HilbertSquare}.

\begin{theorem}[=\Cref{thm:NaturalHoloInv}]
	\label{thm:main:NaturalHoloInv}
	Let $n\geq 2$. Let $S$ be a smooth projective surface and $\sigma$ a holomorphic involution. Assume that $H^1(S, \F2)=0$. Then the induced involution on $S^{[n]}$ is maximal if and only if $\sigma$ is maximal and acts on $H^2(S, \ZZ)$ trivially.
\end{theorem}

\begin{rmk}
	As a somewhat surprising consequence of \Cref{thm:main:NaturalAntiHoloInv} and \Cref{thm:main:NaturalHoloInv}, given a (holomorphic or anti-holomorphic) involution on a smooth projective surface $S$ with $H^1(S, \F2)=0$, the maximality of the naturally induced involution on $S^{[n]}$ for one $n\ge 2$ implies the maximality for all $n\ge 2$.
\end{rmk}

\begin{rmk}
The projectivity assumptions in \Cref{thm:main:NaturalAntiHoloInv} and \Cref{thm:main:NaturalHoloInv} can probably be dropped. They are inherited from \Cref{thm:LiQinWang-IntegralBasis} and \Cref{thm:LiQinWang-Generators} that we use in the proof. 
\end{rmk}

\subsection{Outline of the article}

The structure of the article is the following. \Cref{sec:prelim_on_maximality} contains preliminaries about maximal involutions, including sufficient and necessary conditions for maximality that are exploited throughout the paper. In \Cref{sec:integral_coh_Hilbert_schemes} we recollect facts about the basis given by Nakajima and Li--Qin--Wang for the integral cohomology of the Hilbert scheme of points on a smooth projective surface with vanishing $H^1(-, \F2)$. \Cref{sec:HK-LLV} is an introduction to hyper-Kähler manifolds with a particular focus on manifolds of K3\(^{[n]}\)-type, their monodromy representation on the cohomology of degree 4, and the monodromy representation in terms of the integral basis for Hilbert schemes of points on K3 surfaces introduced in \Cref{sec:integral_coh_Hilbert_schemes}. \Cref{sec:surfaces_with_free_invol} is dedicated to a proof of the fact that for a surface with \(H^1(S,\mathbb{F}_2)=0\) and a \textit{free} holomorphic or anti-holomorphic involution, the induced involution on the Hilbert scheme of points is not maximal. The proof relies on tools of algebraic topology, such as the Smith--Gysin exact sequence and the Kalinin spectral sequence, together with the classification of complex surfaces. The 
 content of \Cref{sec:Natural} is the proofs of \Cref{thm:main:NaturalAntiHoloInv} and \Cref{thm:main:NaturalHoloInv}. These results are achieved using the necessary and sufficient conditions of \Cref{sec:prelim_on_maximality} together with the facts about the integral basis for the cohomology of Hilbert schemes of points of \Cref{sec:integral_coh_Hilbert_schemes}. \Cref{sec:non_existence_maximal_K3[n]} contains the proof of the non-existence of maximal brane involutions for hyper-Kähler manifolds of K3\(^{[n]}\)-type, namely
 the proof of \Cref{thm:main:NonMaxHKnOdd}, \Cref{cor:main:NoMaxRealStructureHKnOdd} and \Cref{thm:main:NonMaxHKnSymp}. The main ingredients of the proofs are the necessary and sufficient conditions for maximality from \Cref{sec:prelim_on_maximality} and the monodromy representation on the degree 4 cohomology from \Cref{sec:prelim_on_HK}, together with some lattice theory.
Some final comments and questions are collected in \Cref{sec:final_comments}.

\subsection*{Convention}
Throughout the paper, we denote by $G$ the cyclic group of order 2 with generator $\sigma$. All topological spaces with involution are considered as $G$-spaces and are assumed to be $G$-homotopy equivalent to a finite $G$-$\mathrm{CW}$-complex (which is the case, for example, of smooth involutions on smooth manifolds having finitely generated cohomology with $\F2$-coefficients).
For a compact manifold, the Poincar\'e duality will be used implicitly to identify cohomology and homology, as well as the associated functorialities (e.g. Gysin maps). For a lattice \(\Lambda\), which is always assumed to be non-degenerate, we denote by \(A_\Lambda:= \Lambda^\vee/\Lambda\) its discriminant group. 

\subsection*{Acknowledgments} 
The $n=2$ case of \Cref{thm:main:NonMaxHKnOdd} was obtained earlier by the fourth author V. Kharlamov, who has benefited from discussions with Rares R\u asdeaconu. We all thank R\u asdeaconu for his contribution. 
We also thank Olivier Benoist for helpful conversations. 

\section{Preliminaries on maximality of involutions}\label{sec:prelim_on_maximality}

Given such a space $X$ with an involution $\sigma$, we let \(X_G:=(X\times EG)/G\) be the Borel construction, where \(G\) acts diagonally (hence freely), \(\mathsf BG\) is the classifying space of \(G\) and \(EG\) is its universal cover. The equivariant cohomology $H^*_G(X,-)$ of $X$ is defined by \(H^*_G(X,-)=H^*(X_G,-)\).

By construction, we have a fibration $X_G\to \mathsf BG$ with fiber $X$, the associated Leray--Serre spectral sequence is the following: 
\begin{equation}
\label{eqn:LeraySerreSS}
E_2^{p,q}=H^p(G, H^q(X, \F2))\Rightarrow H^{p+q}_G(X, \F2).
\end{equation}

The notion of maximality of $\sigma$ has several cohomological characterizations.

\begin{prop}[{\cite[Chapter III, Proposition 4.16]{TomDieck}}]
	\label{prop:MaximalityViaCohomology}
	The following conditions are equivalent:
	\begin{enumerate}
		\item The involution $\sigma$ is maximal, that is, the equality holds in \eqref{eqn:SmithThomInequality}.
		\item The natural morphism $H^*_G(X, \F2)\to H^*(X, \F2)$ is surjective.
		\item $G$ acts trivially on $H^*(X, \F2)$  and the Leray--Serre spectral sequence 
		\begin{equation}
			E_2^{p,q}=H^p(G, H^q(X, \F2))\Rightarrow H^{p+q}_G(X, \F2)
		\end{equation}
		degenerates at $E_2$.
	\end{enumerate}
\end{prop}

The following classification of integral representations of involutions is going back to
Comessatti \cite{Comessatti-RealAVI,Comessatti-RealAVII}
(for a sketch of proof see, for example, \cite[Chapter I, 
Lemmas 3.5, 3.5.1]{Silhol-Surface-LNM}).

\begin{lemma}
\label{lemma:StructureInvolModules}
	Let $M$ be a free abelian group of finite rank equipped with an involution $\sigma$. Then 
	\begin{equation}\label{eqn:StrucutreInvolModules}
		M\cong M_1\oplus M_2\oplus B_1\oplus\cdots\oplus B_{\lambda}
	\end{equation}
	as a $G$-module, where $\sigma|_{M_1}=\id$, $\sigma|_{M_2}=-\id$, and $\sigma|_{B_i}$ has matrix $\begin{pmatrix}
	0 &1\\
	1& 0
	\end{pmatrix}$ for each $i$.
	
\end{lemma}

\begin{rmk}
	\label{rmk:Comessatti0}
	The number $\lambda$ in the standard form (\ref{eqn:StrucutreInvolModules}) is called the \textit{Comessatti characteristic} of $(M,\sigma)$. By \Cref{lemma:StructureInvolModules}, it can be equivalently defined as 
	\begin{equation}
		\lambda(M, \sigma)= \dim_{\F2}\operatorname{Im}\left(1+\sigma\colon M\otimes_{\ZZ}\F2 \to M\otimes_{\ZZ}\F2\right).
	\end{equation}
\end{rmk}

Combined with \Cref{prop:MaximalityViaCohomology}, the following lemma provides the main obstruction that we will exploit towards the maximality of involutions.

\begin{lemma}
	\label{lemma:EquivalentConditionsForSplitting}
	Let $M$ be a free abelian group of finite rank equipped with an involution $\sigma$. Then the following conditions are equivalent.
	\begin{enumerate}
		\item  $\sigma$ acts on $M\otimes \F2$ trivially.
		\item The Comessatti characteristic vanishes: $\lambda(M, \sigma)=0$.
		\item $M=M^\sigma\oplus M^{\sigma-}$, where $M^{\sigma}:=\{x\in M~|~\sigma(x)=x\}$ and $M^{\sigma-}:=\{x\in M~|~\sigma(x)=-x\}$,
		\item For any element $x\in M$, $x+\sigma(x)$ is divisible by 2 in $M$.
      \end{enumerate}
\end{lemma}

\begin{rmk}
	Let $M$ be a finite type torsion-free abelian group with an involution. If the Comessatti characteristic of $M$ is zero, then for any $n\geq 1$, $\Sym^n(M)$ equipped with the naturally induced involution, is again of Comessatti characteristic zero. 
\end{rmk}

\begin{lemma}
	\label{lemma:ComessattiPrimitiveSubModule}
	Let $M$ be a finite type free abelian group with involution $\sigma$
    and $M'\subset M$ a $\sigma$-invariant subgroup.
    Assume that  $M/M'$ is 2-torsion-free, 
    then  $\lambda(M')\leq \lambda(M)$.
\end{lemma}
\begin{proof}
	The short exact sequence 
	\begin{equation}
		0\to M'\to M\to M/M'\to 0
	\end{equation}
	gives rise to an exact sequence
	\begin{equation}
		\operatorname{Tor}^{\ZZ}(M/M', \F2) \to M'\otimes \F2\to M\otimes \F2\to M/M'\otimes \F2\to 0.
	\end{equation}
	The first term vanishes by the assumption that $M/M'$ is 2-torsion-free. Hence $M'\otimes \F2$ is a subspace of $M\otimes \F2$, which is preserved by $\sigma$. Therefore 
	\begin{equation}
		\lambda(M')= \dim_{\F2}\operatorname{Im}\left(1+\sigma|_ {M'\otimes\F2 }\right)\leq  \dim_{\F2}\operatorname{Im}\left(1+\sigma|_ {M\otimes\F2 }\right)=\lambda(M).
	\end{equation}
\end{proof}

\begin{lemma}
    \label{lemma:ComessattiQuotient}
    Let $M, M"$ be free abelian groups of finite type equipped with involution. Let $M\to M''$ be an equivariant surjective homormoprhism. Then $\lambda(M)\geq \lambda(M")$.
\end{lemma}
\begin{proof}
    We denote both involutions by $\sigma$. Consider the following commutative diagram:
    \begin{equation}
        \xymatrix{
        M\otimes \F2 \ar@{->>}[r] \ar[d]^{1+\sigma} & M"\otimes \F2 \ar[d]^{1+\sigma}\\
        M\otimes \F2 \ar@{->>}[r] & M"\otimes \F2 
        }
    \end{equation}
    Since the horizontal maps are surjective by assumption, the rank of the left vertical map is at least the rank of the right vertical map.
\end{proof}

\section{Integral cohomology of Hilbert schemes of points on surfaces}\label{sec:integral_coh_Hilbert_schemes}

In this section, we collect results on the cohomology of Hilbert schemes of points on surfaces that we will need in the sequel.

For a smooth projective complex surface $S$, we denote by $1_S\in H^0(S, \ZZ)$ its fundamental class (with the natural orientation), by $\pt\in H^4(S, \ZZ)$ the class of a point. Let $|0\rangle\in H^0(S^{[0]}, \ZZ)$ be the positive generator, which is the highest weight vector for Nakajima's representation of the Heisenberg Lie algebra on $\bigoplus_{n\geq 0} H^*(S^{[n]}, \QQ)$.

In the rest of the paper, cohomological correspondences will be frequently used. Let us recall the definition, which works for any choice of coefficients ({\it{cf.}} \cite[Section 8.1]{Nakajima}). Given two smooth projective varieties $X, Y$ and a cohomology class $\gamma\in H^*(X\times Y)$, we define
\begin{align*}
        \gamma_*\colon H^*(X)&\to H^*(Y)\\
        v&\mapsto p_{Y,*}(p_X^*(v)\smile \gamma),
\end{align*}
and 
\begin{align*}
        \gamma^*\colon H^*(Y)&\to H^*(X)\\
        v&\mapsto p_{X,*}(p_Y^*(v)\smile \gamma),
\end{align*}
where the push-forward is the Gysin map (using Poincar\'e duality).
The composition of two correspondences $\gamma\in H^*(X\times Y)$, $\zeta\in H^*(Y\times Z)$ is defined as $\zeta\circ \gamma:= p_{XZ,*}(p_{XY}^*(\gamma)\smile p_{YZ}^*(\zeta))$. We have associativity of compositions, and natural equalities $(\zeta\circ \gamma)_*=\zeta_*\circ \gamma_*$ and $(\zeta\circ \gamma)^*=\gamma^*\circ \zeta^*$.

\subsection{Nakajima operators and Li--Qin--Wang integral operators}
\label{subsec:Nakajima-Li-Qin-Wang-Operators}

Let $S$ be a smooth projective surface over $\CC$.
Following \cite[Definition 3.1]{QinWang05}, an operator in $\End\left(\bigoplus_n H^*(S^{[n]},\QQ)\right)$ is called \textit{integral}, if it respects the subgroup 
$\bigoplus_n H^*(S^{[n]},\ZZ)/\tors\subset \bigoplus_n H^*(S^{[n]},\QQ)$. This notion is slightly weaker than the usual one (elements in $\End(\bigoplus_n H^*(S^{[n]},\ZZ))$).

\begin{enumerate}
	\item Nakajima \cite{Nakajima} defined some cohomological operations using natural correspondences. Let us recall the definition.  
	For any $k\geq 0$ and any $\alpha\in H^*(S,\ZZ)$, the (creation) Nakajima operator $\p_{-k}(\alpha)$, for any $j\geq 0$, 
	\begin{equation}
		\p_{-k}(\alpha)\colon H^*(S^{[j]},\ZZ)\to  H^*(S^{[j+k]},\ZZ),
	\end{equation}
	sends an element $\beta\in H^*(S^{[j]},\ZZ)$ to the class $q_*(p^*(\beta)\smile r^*(\alpha)\smile [Q_{j+k, j}])\in H^*(S^{[j+k]},\ZZ)$, where $$Q_{j+k, j}:=\{(Z, x, Z')\in S^{[j]}\times S\times S^{[j+k]}~|~ \operatorname{supp}(I_Z/I_{Z'})=\{x\}\},$$ and $p, r, q$ are the projections from $S^{[j]}\times S\times S^{[j+k]}$ to $S^{[j]}$, $S$ and $S^{[j+k]}$ respectively.
	In particular, taking $j=0$, then $\p_{-k}(\alpha)|0\rangle$ is the image of $\alpha$ via the correspondence $[\Gamma_k]_*\colon H^*(S,\ZZ)\to H^*(S^{[k]},\ZZ)$, where		
	$\Gamma_k:=\{(x, Z)\in S\times S^{[k]}~|~ \operatorname{supp}(Z)=\{x\}\}$.
	In particular, $\p_{-1}(\alpha)|0\rangle=\alpha$.
	\item For any $k\geq 0$, Li--Qin--Wang \cite[Definition 4.1]{LiQinWang-Crelle03}, \cite[(2.7)]{LiQin08} defined the operator $\1_{-k}$ as $\1_{-k}:=\frac{1}{k!}\mathfrak{p}_{-1}(1_S)^k$. Despite of the apparent denominator, in \cite[Lemma 3.3]{QinWang05} it is proved that $\1_{-k}$ is an \textit{integral} operator.
Indeed, it can be equivalently defined as follows: for any $j\geq 0$, the operator $$\1_{-k}: H^*(S^{[j]}, \ZZ)\to H^*(S^{[j+k]}, \ZZ)$$ is the correspondence $[S^{[j, j+k]}]_*$, where $S^{[j, j+k]}:=\{ (Z, Z') \in S^{[j]}\times S^{[j+k]}~|~Z\subset Z'\}$ is the nested Hilbert scheme. 
	\item The operator $\mathfrak{m}$ is defined in \cite[\S 4.2]{QinWang05} and in \cite[Definition 3.3]{LiQin08}. More precisely, for $\lambda$ a partition of $n$, $\alpha\in H^2(S, \ZZ)/\tors$, we have an operator
    \begin{equation}
        \mathfrak{m}_{\lambda} (\alpha)\colon H^*(S^{[k]}, \QQ) \to H^*(S^{[k+n]},\QQ),
    \end{equation}
    which is proven to be an integral operator by Li--Qin in \cite[Theorem 3.6]{LiQin08}.
    The precise definition of $\mathfrak{m}$ is somewhat involved. Let us only mention here that when $\alpha=[C]$ for a smooth irreducible curve $C$ in $S$, then for any partition $\lambda=(\lambda_1\geq \cdots \geq \lambda_N)$ of $n$, we have $\mathfrak{m}_{\lambda} ([C])|0\rangle= [L^\lambda C]$, where $L^\lambda C$ is the closure in $S^{[n]}$ of $\{\lambda_1x_1+\cdots+\lambda_Nx_N ~|~ x_i \in C \text{ distinct}\}.$ 
\end{enumerate}

\subsection{Integral basis}
Thanks to the work of Li--Qin--Wang \cite[Theorem 1.1]{LiQin08} and \cite[Theorem 1.1]{QinWang05}, for a projective surface with vanishing odd Betti numbers, we have a concrete integral basis for the cohomology of its punctual Hilbert schemes.

\begin{theorem}[Li--Qin--Wang]
	\label{thm:LiQinWang-IntegralBasis}
	Let $S$ be a smooth projective surface with $b_1(S)=0$. Let $\alpha_1, \cdots, \alpha_k$ be an integral basis of $H^2(S, \ZZ)/\tors$. Then the following classes form an integral basis of $H^*(S^{[n]}, \ZZ)/\operatorname{tors}$:
	\begin{equation}
		\frac{1}{\z_\lambda} \p_{-\lambda}(1_S)\p_{-\mu}(\pt) \m_{\nu^1}(\alpha_1)\cdots \m_{\nu^k}(\alpha_k)|0\rangle  
	\end{equation}
	where $\lambda, \mu, \nu^1, \cdots, \nu^k$ run through all partitions satisfying $|\lambda|+|\mu|+ |\nu^1|+\cdots |\nu^k|=n$, and for a partition $\lambda=(\lambda_1\geq \lambda_2\geq\cdots\geq \lambda_l)=(1^{m_1}2^{m_2}\cdots r^{m_r})$, let $|\lambda|:=\sum im_i=\sum_i\lambda_i$,
	$\z_{\lambda}:=\prod_{i} (i^{m_i}m_i!)$, and \(\mathfrak{p}_{-\lambda}=\prod_i \mathfrak{p}_{-\lambda_i}\).
\end{theorem}

\begin{rmk}[Integral basis in $H^2$ and $H^4$]
	\label{rmk:IntegralBasisH2H4}
	What is particularly important for us is the collection of basis elements in degree 2 and degree 4, which we list below for the easy of later reference ($n\geq 2$):
	\begin{enumerate}
		\item Integral basis of $H^2(S^{[n]}, \ZZ)/\tors$:
		\begin{equation}
        \label{eqn:delta1}
			\1_{-(n-1)}\alpha_1,\quad \cdots, \1_{-(n-1)}\alpha_k, \quad \delta:=\frac{1}{2}\1_{-(n-2)}\p_{-2}(1_S)|0\rangle.
		\end{equation}
		This amounts to the well-known isomorphism 
		\begin{equation}
        \label{eqn:delta2}
			H^2(S^{[n]}, \ZZ)/\tors\cong H^2(S, \ZZ)/\tors \oplus \ZZ \delta,
		\end{equation}
		where $\delta$ is the half of the exceptional divisor class, and the isomorphism identifies a class $\alpha\in H^2(S, \ZZ)/\tors$ with $\1_{-(n-1)}\alpha=([S^{[1, n]}])_*(\alpha)\in H^2(S^{[n]}, \ZZ)/\tors$. 
		\item Integral basis of $H^4(S^{[n]}, \ZZ)/\tors$:
		\begin{align}
			&\1_{-(n-1)}\pt;&\\
			&\1_{-(n-2)}\p_{-2}(\alpha_i)|0\rangle,& \text{ with } 1\leq i\leq k ;\\
			&\1_{-(n-2)} \p_{-1}(\alpha_i)\p_{-1}(\alpha_j)|0\rangle, &\text{  with  } 1\leq i<j\leq k;\\
			&\1_{-(n-2)}\m_{1,1}(\alpha_i)|0\rangle, &\text{ with } 1\leq i\leq k ;\\
			&\frac{1}{3}\1_{-(n-3)}\p_{-3}(1_S)|0\rangle,& \text{ when $n \geq 3$};\\
			&\frac{1}{2}\1_{-(n-3)}\p_{-2}(1_S)\alpha_i, & \text{ with $1\leq i\leq k$,  when $n \geq 3$};\\
			&\frac{1}{8}\1_{-(n-4)}\p_{-2}(1_S)\p_{-2}(1_S)|0\rangle, & \text{ when $n \geq 4$}.
		\end{align}
	\end{enumerate}
\end{rmk}

\begin{rmk}[Key relation]
	\label{rmk:KeyRelation}
	The classes of the form $\1_{-(n-2)} \p_{-1}(\alpha)^2|0\rangle \in H^4(S^{[n]}, \ZZ)$ do not appear in the above list, since we have the following relation, which will play a crucial role in the proof of our main results.
	\begin{equation}
		\1_{-(n-2)}\m_{1,1}(\alpha)|0\rangle=\frac{1}{2}\left(\1_{-(n-2)} \p_{-1}(\alpha)^2|0\rangle - \1_{-(n-2)}\p_{-2}(\alpha)|0\rangle\right).
	\end{equation}
\end{rmk}

\subsection{Integral generators}

We will need the following set of integral generators for the cohomology of Hilbert schemes by Li and Qin \cite[Theorem 1.2]{LiQin08}.
\begin{theorem}[Li--Qin]
	\label{thm:LiQinWang-Generators}
	Let $S$ be a smooth projective complex surface with $b_1(S)=0$, and $n$ a positive integer. Then $H^*(S^{[n]}, \ZZ)/\tors$ is generated by the following three types of elements:
	\begin{enumerate}
		\item[(i)] For $1\leq j\leq n$, the Chern class $c_j(\mathcal{O}_S^{[n]})\in H^{2j}(S^{[n]}, \ZZ)$, where $\mathcal{O}_S^{[n]}$ is the tautological rank-$n$ bundle defined as $p_*(\mathcal{O}_{\mathcal{Z}_n})$ where $p\colon S^{[n]}\times S\to S^{[n]}$ is the projection and $\mathcal{Z}_n\subset S^{[n]}\times S$ is the universal codimension-2 subscheme given by $\{(\xi, x)\in S^{[n]}\times S~|~x\in \operatorname{supp}(\xi)\}$.
		\item[(ii)] For $1\leq j\leq n$, and $\alpha\in H^2(S, \ZZ)/\tors$,  the class $\mathds{1}_{-(n-j)}\mathfrak{m}_{(1^j)}( \alpha)|0\rangle$.
		\item[(iii)] For $1\leq j\leq n$,  the class $\mathds{1}_{-(n-j)}\mathfrak{p}_{-j}(\pt)|0\rangle$, where $\pt\in H^4(S, \ZZ)$ denotes the class of a point. 
	\end{enumerate}
\end{theorem}

\section{Hyper-K\"ahler manifolds and monodromy action}\label{sec:prelim_on_HK}
\label{sec:HK-LLV}
In this section, we recall some basic features of compact hyper-K\"ahler manifolds. We remind known facts about the monodromy group of hyper-K\"ahler manifolds of K3\(^{[n]}\)-type and analyze in detail the monodromy action on the cohomology of degree 4 when \(n\geq 4\). We also describe the monodromy representation in terms of the integral basis introduced in the previous section for Hilbert schemes of points on a K3 surface. 

\subsection{Basics on compact hyper-K\"ahler manifolds}
By definition, a compact K\"ahler manifold  $X$ is called \textit{hyper-K\"ahler} if it is simply-connected and $H^0(X, \Omega_X^2)$ is generated by a symplectic (i.e.~nowhere degenerate) holomorphic 2-form. In particular, $\dim(X)$ is even, and the canonical bundle of $X$ is trivial. From the viewpoint of differential geometry, such manifolds are characterized by the existence of the so-called hyper-K\"ahler metrics, that is, a Ricci-flat metric whose holonomy group is the compact symplectic group $\operatorname{Sp}(n)$, where $n=\frac{\dim(X)}{2}$. See \cite{Beauville} and \cite{HuyInventiones} for basic definitions.

Compact hyper-K\"ahler manifolds are generalizations of K3 surfaces and by the Beauville--Bogomolov decomposition theorem \cite{Beauville}, they form one of the three basic types of building blocs for compact K\"ahler manifolds with vanishing first Chern class. Higher-dimensional examples include Hilbert schemes of K3 surfaces, generalized Kummer varieties, O'Grady's 10-dimensional (resp.~6-dimensional) crepant resolutions of certain moduli spaces of semistable sheaves on K3 (resp.~abelian) surfaces, and deformations of these. By the Bogomolov--Tian--Todorov theorem, the deformation space of a compact hyper-K\"ahler manifold is smooth. 

An essential part of the geometry of a compact hyper-K\"ahler manifold is controlled by its second cohomology. More precisely, for a compact hyper-K\"ahler manifold $X$, $H^2(X, \ZZ)$  can be endowed with a natural quadratic form, the Beauville--Bogomolov--Fujiki (BBF) form \cite{Beauville}, that is compatible with the Hodge structure on $H^2(X,\ZZ)$ and depends only on the topology of \(X\). From these data, one can define the period domain, and the corresponding period map from the Kuranishi space to the period domain is \'etale (local Torelli theorem). A global Torelli theorem is proved by Verbitsky \cite{VerbitskyTorelli}, see also \cite{Markman-SurveyTorelli}, \cite{HuybrechtsTorelli}.

Given two compact hyper-K\"ahler manifolds $X, X'$ that are deformation equivalent, 
a ring isomorphism $$\phi\colon H^*(X, \ZZ) \to H^*(X', \ZZ)$$ is called a \textit{parallel transport operator}, if there exist a smooth and proper family 
$\pi\colon \mathscr{X}\to B$ of compact hyper-K\"ahler manifolds over an analytic base 
$B$, two points $b, b'\in B$ with isomorphisms $\psi\colon X\xrightarrow{\simeq} \mathscr{X}_{b}$ and  $\psi'\colon X'\xrightarrow{\simeq } \mathscr{X}_{b'}$, and a continuous path $\gamma$ from $b$ to $b'$ such that
$(\psi')^*\circ \gamma_*\circ (\psi^*)^{-1}=\phi$,
where $\gamma_*$ is the parallel transport in the local system $\mathsf{R}\pi_*\ZZ$; see \cite[Definition 1.1]{Markman-SurveyTorelli}.

\subsection{Hyper-K\"ahler rotation}

Let $X$ be a compact hyper-K\"ahler manifold equipped with a hyper-K\"ahler metric $g$. Then there is a space of compatible complex structures parametrized by a 2-dimensional sphere. More precisely, there are three complex structures $I, J , K$ on $X$ satisfying the quaternion relations :
	\begin{equation}
    \label{eqn:QuaternionRelation}
		I^2=J^2=K^2=-\id, \quad IJ=-JI=K, \quad JK=-KJ=I,\quad KI=-IK=J,
	\end{equation}
and for any $a, b,c\in \RR$ with $a^2+b^2+c^2=1$, $aI+bJ+cK$ is a complex structure compatible with $g$. This freedom of choices of complex structures, known as \textit{hyper-K\"ahler rotation}, allows us to switch between real structures and anti-symplectic holomorphic involutions on $X$.

\begin{prop}
\label{prop:HKRotation}
    Let $X$ be a compact hyper-K\"ahler manifold. Let $\sigma$ be a real structure, i.e.~an anti-holomorphic involution on $X$. Then there exists a new complex structure on $X$ with respect to which $\sigma$ is holomorphic anti-symplectic. 
\end{prop}
\begin{proof}
     Fix a $\sigma$-invariant hyper-K\"ahler metric on $X$. Let $I$ be the complex structure on $X$. As $\sigma$ is anti-holomorphic, we have $$\sigma^*(I)=-I.$$ 
    Let $\{I, J, K\}$ be a triple of complex structures on the real tangent bundle of $X$ satisfying the quaternion relations \eqref{eqn:QuaternionRelation}.
	Let $\omega_I$, $\omega_J$ and $\omega_K$ be the corresponding K\"ahler forms.
    
        Let $\sigma^*(J)=\lambda I+aJ+bK$, where $\lambda, a, b\in \RR$ such that $\lambda^2+a^2+b^2=1$. Since $IJ=-JI$, we have $\sigma^*(I)\sigma^*(J)=-\sigma^*(J)\sigma^*(I)$. Therefore the condition $\sigma^*(I)=-I$ implies that $\lambda=0$, hence
       $$\sigma^*(J)=a J+ bK$$ for some $a, b\in \RR$ with $a^2+b^2=1$.
        It yields that $\sigma^*(K)=\sigma^*(I)\sigma^*(J)=(-I)(aJ+bK)=bJ-aK$.
    
	In other words, the restriction of $\sigma^*$ to the plane $\RR J\oplus \RR K$ is a reflection. Up to changing $J, K$ by a rotation of this plane, one can assume that $\sigma^*(J)=-J$ and $\sigma^*(K)=K$. 
	
	Now we equip the manifold $X$ with the new complex structure $K$, then $\sigma$ is holomorphic and the holomorphic symplectic form  $$\eta_K=\omega_I+\sqrt{-1}\omega_J$$
	is clearly $\sigma$-anti-invariant. 
\end{proof}

\subsection{The monodromy group of manifolds of K3\(^{[n]}\)-type}
Let \(X\) be a compact hyper-K\"ahler manifold.
The monodromy group \(\Mon(X)\) is the subgroup of \(\prod_i\GL(H^i(X,\mathbb{Z}))\) consisting of parallel transport operators from \(X\) to itself. The group \(\Mon(X)\) acts on \(H^*(X,\mathbb{Z})\) by degree-preserving ring isomorphisms. We denote by \(\Mon^2(X)\) the subgroup of \(O(H^2(X,\mathbb{Z}))\) obtained by restricting parallel transport operators to \(H^2(X,\mathbb{Z})\); 
in other words, \(\Mon^2(X)\) is the image of \(\Mon(X)\) under the projection \(\prod_i\GL(H^i(X,\mathbb{Z}))\to \GL(H^2(X, \mathbb{Z})\). The isomorphism classes of the groups \(\Mon(X)\), \(\Mon^2(X)\) depend only on the deformation class of \(X\).

For a K3 surface \(S\), we have the following computation of its monodromy group; see for example \cite[Chapter 7, Proposition 5.5]{HuyK3}:
\[\Mon^2(S)=O^+(H^2(S,\mathbb{Z}))\subset O(H^2(S,\mathbb{C})),\]
where $O^+(H^2(S,\mathbb{Z}))$ stands for the isometries of the lattice $H^2(S,\mathbb{Z})$ with spinor norm 1 (i.e.~preserving orientation in 
positive definite 3-dimensional subspaces of $H^2(S,\mathbb{R})$). 
Note that $\Mon^2(S)$ is not an algebraic closed subgroup and its Zariski closure is the entire \(O(H^2(S,\mathbb{C}))\).

Let \(X\) be a K3$^{[n]}$-type hyper-K\"ahler manifold. Thanks to \cite[Lemma 2.1]{Markman-ExistenceUniversalFamily}, the natural morphism $\Mon(X)\to \Mon^2(X)$ is an isomorphism. From results of Markman \cite[Theorem 1.2 and Lemma 4.2]{markman2010integral} 
we know that
\[\Mon^2(X)=W(H^2(X,\mathbb{Z})),\]
where $W(H^2(X,\mathbb{Z}))$ stands for the subgroup of isometries in \(O^+(H^2(X,\mathbb{Z}))\) acting as \(\pm\id\) on the discriminant group of the lattice 
$H^2(X,\mathbb{Z})$. Here and in the sequel, $H^2(X, \ZZ)$ is always equipped with the Beauville--Bogomolov--Fujiki quadratic form.
We have therefore the \textit{discriminant character} 
\begin{equation}
    \tau\colon \Mon^2(X)\to \{\pm1\}
\end{equation} that records the action on the discriminant.

Markman proved in \cite[Lemma 4.11]{Markman_OnTheMonodromy} (combined with Lemma 4.10 in {\it loc.~cit.}) that the Zariski closure of the subgroup \(\Mon(X)\subset \GL(H^*(X,\mathbb{C}))\) is isomorphic to \(O(H^2(X,\mathbb{C}))\times \{\pm 1\}\) if \(X\) is of K3\(^{[n]}\)-type with \(n\geq 3\), and isomorphic to \(O(H^2(X,\mathbb{C}))\) if \(X\) is of K3\(^{[2]}\)-type\footnote{The K3$^{[2]}$-type case is not stated in \cite[Lemma 4.11]{Markman_OnTheMonodromy}, but it follows from the fact that in this case, the discriminant group \(A_{H^2(X, \mathbb{Z})}\) of $H^2(X, \mathbb{Z})$ is isomorphic to $\mathbb{Z}/2\mathbb{Z}$ equipped with a non-zero quadratic form, hence $O(A_{H^2(X, \mathbb{Z})})$ is trivial. In particular, the discriminant character $\tau$ is trivial.}.
By this result of Zariski closure, we obtain that for any \(n\geq 1\) and \(X\) of K3\(^{[n]}\)-type, a linear representation 
\begin{equation}\label{eq:monodromy_representation_from_H2}
    \rho\colon O(H^2(X,\mathbb{C}))\times \{\pm 1\}\to \GL(H^*(X,\mathbb{C}))
\end{equation}
acting by degree-preserving ring isomorphisms (the action of $\{\pm 1\}$ is set to be trivial if $n<3$).

\subsection{Monodromy representation on the degree-4 cohomology}
\label{sec:HK-Monodromy}
Let $X$ be a compact hyper-K\"ahler manifold of $\K3^{[n]}$-type with $n\geq 4$.
Following Markman \cite{markman2010integral}, we define\footnote{It is denoted by $Q^4(X, \ZZ)$ in \cite{markman2010integral}.}
\begin{equation}
    Q(X, \ZZ):=H^4(X, \ZZ)/\Sym^2H^2(X, \ZZ).
\end{equation}
Note that $\Sym^2H^2(X,\ZZ)$ is preserved by the monodromy action, 
hence we have a natural action of $\Mon(X)$ on $Q(X,\ZZ)$.
For any monodromy operator $\sigma\in \Mon(X)\subset \GL(H^*(X,\ZZ))$, we denote by 
\begin{equation}
    \sigma_2\in \Mon^2(X)\subset O(H^2(X, \ZZ)) \quad\quad \text{ and } \quad\quad \sigma_4\in \GL(H^4(X, \ZZ))\quad\quad \text{ and } \quad\quad \sigma_Q\in O(Q(X,\ZZ))
\end{equation}
the restricted/induced automorphisms on $H^2(X, \ZZ)$, $H^4(X, \ZZ)$ and $Q(X,\ZZ)$ respectively.

We need the following result of Markman \cite[Theorem 1.10]{markman2010integral} on $Q(X, \ZZ)$. Let \(\overline{c}_2(X)\) denote the image of the second Chern class $c_2(X)$ under the natural projection $H^4(X,\ZZ)\to Q(X,\ZZ)$.

\begin{theorem}[Markman]
\label{thm:Markman-Q}
Notation is as above. Let $X$ be a $\K3^{[n]}$-type hyper-K\"ahler manifold with $n\geq 4$. Then
\begin{enumerate}
    \item $Q(X,\ZZ)$ is a free abelian group equipped with a natural Hodge structure and a monodromy invariant quadratic form $q_Q$, such that the resulting lattice is the Mukai lattice $E_8(-1)^{\oplus 2}\oplus U^{\oplus 4}$.
    \item $\frac{1}{2}\overline{c}_2(X)$ is a non-zero primitive element in $Q(X, \ZZ)$ with square $q_Q(\frac{1}{2}\overline{c}_2(X))=2n-2$.
    \item There is a Hodge isometry \(e\colon H^2(X,\mathbb{Z})\xrightarrow{\cong} \overline{c}_2(X)^\perp\subset Q(X,\mathbb{Z})\) such that the transformation of $e$ under the monodromy group is given by the discriminant character $\tau\colon \Mon^2(X)\to \{\pm 1\}$. More precisely, the following diagram is commutative:
    \begin{equation}
    \label{diag:MonQisMon2Twisted}
        \xymatrix{
        \Mon(X) \ar[d] \ar[r]^-{\rho} & O(\overline{c}_2(X)^{\perp}, q_Q)\ar[d]\\
        \Mon^2(X) \ar[r]^-{\tau\cdot\rho}& O(H^2(X,\ZZ), q)
        }
    \end{equation}
    where the top arrow is the monodromy representation, the bottom arrow is the monodromy representation multiplied by the discriminant character,
    and the right vertical arrow sends any $F\in O(\overline{c}_2(X)^{\perp}, q_Q)$ to the composition $e^{-1}\circ F\circ e$. 
    In other words, for any $\alpha\in H^2(X, \ZZ)$ and $\sigma\in \Mon(X)$,
    \begin{equation}
    \label{eqn:MonQisMon2Twisted}
        \sigma_Q(e(\alpha))=\tau(\sigma)\cdot e(\sigma_2(\alpha)).
    \end{equation}
\end{enumerate}
\end{theorem}

\begin{rmk}
    In the statement of \cite[Theorem 1.10]{markman2010integral}, it was only indicated that there is some non-trivial character $\tau\colon \Mon^2(X)\to \{\pm 1\}$ making the diagram \eqref{diag:MonQisMon2Twisted} commutative, without determining $\tau$ explicitly. However, it is easy to see that $\tau$ must be the discriminant character: since $Q(X, \ZZ)$ is a unimodular lattice and $\sigma_Q(\overline{c}_2(X))=\overline{c}_2(X)$, $\sigma_Q$ must act trivially on the discriminant of the sublattice $\overline{c}_2(X)^\perp$ in $Q(X, \ZZ)$. Therefore the transported action of $\sigma_Q$ to $H^2(X, \ZZ)$ by conjugating with the isomorphism $e$ must act trivially on the discriminant of $H^2(X,\ZZ)$. This implies that $\tau$ must be the discriminant character.
\end{rmk}

\subsection{Monodromy representation for Hilbert schemes of points on K3 surfaces}\label{section:monodromy representation}

We relate the action of monodromy operators of a K3 surface with the integral basis of the cohomology of the Hilbert schemes described in the previous sections. The results are based on Markman's results and are collected in Oberdieck's work \cite[Section 3.6]{oberdieck2024holomorphic}. We briefly recall them for the reader's convenience.

Let \(S\) be a K3 surface and consider the Hilbert scheme of \(n\) points \(S^{[n]}\). 
As remarked before, (\ref{eq:monodromy_representation_from_H2}) gives for any \(n\geq 1\), a linear representation 
\begin{equation}
\label{eqn:rho_n}
    \rho_n\colon O(H^2(S^{[n]},\mathbb{C}))\times \{\pm 1\}\to \GL(H^*(S^{[n]},\mathbb{C}))
\end{equation}
acting by degree-preserving ring isomorphisms. Notice that there are natural embeddings 
\begin{equation}
\label{eqn:Inclusions}
    O^+(H^2(S,\mathbb{Z}))\subset 
W(H^2(S^{[n]},\mathbb{Z}))\subset O(H^2(S^{[n]},\mathbb{Z}))\times \{\pm 1\},
\end{equation}
where the first is induced by the natural inclusion \(H^2(S,\mathbb{Z})\subset H^2(S^{[n]},\mathbb{Z})\) and the second one is given by \(g\mapsto (g,\tau(g))\).

\begin{rmk}
\label{rmk:Monodromy}
    Via the inclusions \eqref{eqn:Inclusions}, the restriction of the representation $\rho_n$ in \eqref{eqn:rho_n} to $O^+(H^2(S,\mathbb{Z}))=\Mon^2(S)$ is given geometrically by sending a monodromy operator of $S$ defined by a loop $\gamma$ in the base of a family of K3 surfaces $\mathscr{S}\to B$ to the monodromy operator defined by the same loop $\gamma$ for the associated family of Hilbert schemes $\Hilb^n_B\mathscr{S}\to B$.
\end{rmk}

As is recalled in \eqref{eqn:delta1} and \eqref{eqn:delta2}, we have a natural isomorphism \(H^2(S^{[n]},\mathbb{Z})\cong H^2(S,\mathbb{Z})\oplus \mathbb{Z}\cdot \delta\),
where $\delta$ is the half of the class of the exceptional divisor.
Via this isomorphism, we have a canonical injective morphism 
$$O(H^2(S, \CC))\hookrightarrow O(H^2(S^{[n]}, \CC))$$
by extending by the trivial action on $\delta$. In the sequel, we identify
$O(H^2(S, \CC))$ with its image in $O(H^2(S^{[n]}, \CC))$, which is nothing but $O(H^2(S^{[n]}, \CC))_{\delta}$, the stabilizer of $\delta$.

We moreover have the map \(O(H^2(S,\mathbb{C}))\to \GL(H^*(S,\mathbb{C}))\) 
given by \(g \mapsto \tilde{g}:=\id_{H^0(S,\mathbb{C})}\oplus g\oplus\id_{H^4(S,\mathbb{C})}\).

Nakajima operators are encoded in the following linear maps. For any $k\geq 0$,
\begin{equation}
\label{eqn:NakajimaP}
    \mathfrak{p}_{-k}\colon H^*(S, \CC)\to \Hom(H^*(S^{[n]},\mathbb{C}),H^*(S^{[n+k]},\mathbb{C}))
\end{equation}
is defined by \(\mathfrak{p}_{-k}(\alpha)(v)=(r^*\alpha)_* (v):=p_{S^{[n+k]}, *}(p_{S^{[n]}}^*(v)\smile r^*\alpha ),\)
where \(r\colon S^{[n,n+k]}_0\to S\) is the residue map, $p_{S^{[n]}}$ and $p_{S^{[n+k]}}$ are the natural projections from $S^{[n,n+k]}_0$  to $S^{[n]}$ and $S^{[n+k]}$ respectively, and
$S_0^{[n, n+k]}:=\{(Z, Z')\in S^{[n]}\times S^{[n+k]}~|~ Z\subset Z' \text{ and } \operatorname{supp}(\mathcal{I}_{Z}/\mathcal{I}_{Z'}) \text{ is one point}\}$, and \(\alpha\in H^*(S)\).

\begin{lemma}\label{lemma:Zariski_closure}
For any $n, k\in \mathbb{N}$, and any $\alpha \in H^*(S,\CC)$, the group 
	\[M_{n,k}(\alpha):=\left\{g\in O(H^2(S,\mathbb{C}))\mid 
    \mathfrak{p}_{-k}(\tilde{g}(\alpha))\circ \rho_n(g,1)=\rho_{n+k}(g,1)\circ \mathfrak{p}_{-k}(\alpha)\right\}\]
	is a closed algebraic subgroup of \(O(H^2(S,\mathbb{C}))\).
\end{lemma}
\begin{proof}
    It is straightforward to check that $M_{n,k}(\alpha)$ is a subgroup of $O(H^2(S, \CC))$. To show that the condition in the statement is a Zariski closed condition, it suffices to notice the following elementary facts:
    \begin{itemize}
        \item For any $n$, the homomorphism $\rho_n(-, 1)\colon O(H^2(S,\CC))\to \GL(H^*(S^{[n]}, \CC))$ is algebraic. Indeed this follows from the definition of $\rho_n$ as it is defined by extending to algebraic closure of the monodromy group. 
        \item The map sending $g$ to $\widetilde{g}$ is clearly a morphism of algebraic groups.
        \item The map $O(H^*(S,\CC))\to H^*(S,\CC)$ of evaluation at $\alpha$ is an algebraic map.
        \item The map $\p_{-k}$ as in \eqref{eqn:NakajimaP} is algebraic (actually, linear). 
        \item Composition of algebraic maps is algebraic.
    \end{itemize}
    Hence both sides of the condition in the statement are algebraic in $g$, this defines a closed algebraic subgroup.
\end{proof}

\begin{prop}
	For any \(\alpha \in H^*(S,\mathbb{C})\), \(g\in O(H^2(S,\mathbb{C}))\) and \(n,k\in\mathbb{N}\), we have 
    \begin{equation}
        \label{eqn:Nakajima-Monodromy-Commute}
        \mathfrak{p}_{-k}(\tilde{g}(\alpha))\circ \rho_n(g,1)=\rho_{n+k}(g,1)\circ \mathfrak{p}_{-k}(\alpha).
    \end{equation}
\end{prop}
\begin{proof}
    Note that given a monodromy operator on the K3 surface
    \[g\in O^+(H^2(S,\mathbb{Z}))=\Mon^2(S),\]
    the induced monodromy operator on $S^{[n]}$, still denoted by $g\in \Mon^2(S^{[n]})$, has trivial discriminant character:
    $\tau(g)=1$, since it preserves the exceptional divisor, hence also the class $\delta$.
    By \Cref{rmk:Monodromy}, it is straightforward to check that the relation \eqref{eqn:Nakajima-Monodromy-Commute} holds for any $g\in O^+(H^2(S,\mathbb{Z}))=\Mon^2(S)$.
    By \Cref{lemma:Zariski_closure}, the relation also holds for any element in the Zariski closure of \(O^+(H^2(S,\mathbb{Z}))\subset O(H^2(S,\mathbb{C}))\). The Zariski closure is the entire \(O(H^2(S,\mathbb{C}))\), so that \eqref{eqn:Nakajima-Monodromy-Commute} holds for any isometry as in the statement.
\end{proof}

\begin{cor}[Property 3 of {\cite[Section 3.6]{oberdieck2024holomorphic}}]
\label{cor:property_3}
Let $n$ be a positive integer. For any \(g\in O(H^2(S, \CC))=O(H^2(S^{[n]},\mathbb{C}))_\delta\) , any  \(k_1,\dots,k_l\in \mathbb{N}\) with \(n=k_1+\dots+k_l\), and any \(\alpha_1,\dots,\alpha_l\in H^*(S,\mathbb{C})\), we have the following equality in $H^*(S^{[n]}, \CC)$:
\begin{equation}
\label{eqn:CommutationRelation}
    \rho_n(g,1)\mathfrak{p}_{-k_1}(\alpha_1)\dots \mathfrak{p}_{-k_l}(\alpha_l)|0\rangle=\mathfrak{p}_{-k_1}(\tilde{g}(\alpha_1))\dots \mathfrak{p}_{-k_l}(\tilde{g}(\alpha_l))|0\rangle,
\end{equation}
\end{cor}
\begin{proof}
    An iterated application of \eqref{eqn:Nakajima-Monodromy-Commute} gives the commutation rule \eqref{eqn:CommutationRelation}.
\end{proof}

\begin{rmk}
    As a special case of \Cref{cor:property_3}, for any monodromy operator \(g\in \Mon^2(S^{[n]})=\Mon(S^{[n]})\) such that \(g(\delta)=\delta\), we have
    $g(\mathfrak{p}_{-k_1}(\alpha_1)\dots \mathfrak{p}_{-k_l}(\alpha_l)|0\rangle)=\mathfrak{p}_{-k_1}(g(\alpha_1))\dots \mathfrak{p}_{-k_l}(g(\alpha_l))|0\rangle.$
    Indeed, if \(g(\delta)=\delta\) then \(\tau(g)=1\). Hence the monodromy action by $g$ is the same as by $\rho_n(g, 1)$.
\end{rmk}

The following result will be useful to understand the action of a monodromy operator that acts as $-\id$ on the discriminant group $A_{H^2(X,\ZZ)}$.
\begin{lemma}[Property 2 of {\cite[Section 3.6]{oberdieck2024holomorphic}}]\
\label{lemma:Property_2}
    Let $n\geq 3$ be an integer and $S^{[n]}$ be the $n$-th punctual Hilbert scheme of a K3 surface $S$. Then
    $$\rho_n(\id_{H^2(S^{[n]}, \CC)}, -1)=D\circ \rho_n(-\id_{H^2(S^{[n]}, \CC)}, 1),$$
    where $D$ is the degree operator which acts on $H^{2i}(S^{[n]}, \CC)$ by multiplication by $(-1)^i$.
\end{lemma}


\section{Surfaces with a free involution and their Hilbert schemes}\label{sec:surfaces_with_free_invol}

Among real varieties, or more generally, in the study of the geometry of involutions, those without real (resp.~fixed) points often play a distinguished role and sometimes present extra difficulties. It is indeed the case in the proofs of \Cref{thm:main:NaturalAntiHoloInv} and \Cref{thm:main:NaturalHoloInv}. The goal of this section is to prove these theorems in the fixed-point-free case. More precisely, the main result of the section is the following:

\begin{theorem}
\label{thm:NonMaxHilbSurfWithoutFixPoint}
    Let $S$ be  a compact complex surface with $H^1(S, \F2)=0$. Let $\sigma$ be a holomorphic or anti-holomorphic involution of $S$ without fixed point. Then for any $n\geq 1$, the naturally induced involution on the Hilbert scheme $S^{[n]}$ is not maximal. 
\end{theorem}

\Cref{thm:NonMaxHilbSurfWithoutFixPoint} is proved in \Cref{sec:proof_of_NonMaxHilbSurfWithoutFixPoint} and it is obtained by combining \Cref{thm:FixedPointTheorem}, \Cref{cor:FixedPointFree-NonMaximal},  \Cref{rmk:Remaining cases}, \Cref{gen-theorem}, \Cref{cor:HirzebruchHilb}, \Cref{thm:FakeQuadrics}, \Cref{fake-quadric-theorem}.

\subsection{Topological constraints on free involutions of surfaces}
As a first step, we provide a strong restriction satisfied by a fixed-point-free (holomorphic or anti-holomorphic) involution on a compact complex surface.
We start with a lemma from algebraic topology of involutions on manifolds, where we use the following notation.
Let $M$ be a topological space endowed with a continuous involution $\sigma: M\to M$ without fixed point:
\begin{equation}
    M^{\sigma}=\emptyset.
\end{equation}
Assume that $M$ has a CW-complex structure such that $\sigma$ is cellular. Consider the associated Smith--Gysin long exact sequence in the following form (see \cite[Theorem 1.2.1]{RealEnriques} for example):
\begin{equation}
 \label{eqn:SmithSeq}
		\cdots\to H_{r+1}(M/\sigma, \F2)\xrightarrow{\gamma_r} H_r(M/\sigma, \F2)\xrightarrow{\alpha_r} H_r(M, \F2)\xrightarrow{\beta_r} H_r(M/\sigma, \F2)\xrightarrow{\gamma_{r-1}} H_{r-1}(M/\sigma, \F2) \to\cdots   
\end{equation}
and put $I_r=H_r(M,\F2)^{\sigma}$ the subspace of invariant elements. Recall that $\gamma_r=\cap\omega$ where $\omega\in H^1(M/\sigma,\F2)$ is the characteristic class of the double covering $\pi : M\to M/\sigma$, while $\alpha_r=\pi^*$ and $\beta_r=\pi_*$ are the transfer and projection homomorphisms, respectively.
\begin{lemma}\label{InvClasses}
Let $M$ be a connected compact oriented 4-manifold with $H^1(M,\F2)=0$ and $\sigma: M\to M$ an orientation-preserving involution with empty fixed locus $M^\sigma=\emptyset$. Then:
\begin{align}
    & \dim I_2=\dim \Im \alpha_2+1=b_2(M/\sigma, \F2),\label{eqn:I2Dim} \\
    & \dim \Im \alpha_2= \frac12 b_2(M). \label{eqn:RankAlpha2}
\end{align}
\end{lemma}
\begin{proof}
Let the notation be as before. In the Smith--Gysin sequence \eqref{eqn:SmithSeq}:
\begin{itemize}
    \item clearly $\beta_0\colon H_0(M, \F2)\to H_0(M/\sigma, \F2)$ and    $\alpha_4\colon H_4(M/\sigma, \F2)\xrightarrow{\pi^*} H_4(M, \F2)$ are isomorphisms;
    \item $H_3(M,\F2)=H_1(M,\F2)=0$ by the assumption that $H^1(M,\F2)=0$.
\end{itemize}
Therefore the exactness of \eqref{eqn:SmithSeq} implies the following:
\begin{itemize}
    \item $\gamma_3\colon H_4(M/\sigma, \F2)\xrightarrow{\cap \omega}H_3(M/\sigma, \F2)$ is an isomorphism, hence $H_3(M/\sigma, \F2)$ is 1-dimensional and generated by the Poincar\'e dual of $\omega$. 
    \item $\gamma_0\colon H_1(M/\sigma, \F2)\xrightarrow{\cap \omega}H_0(M/\sigma, \F2)$ is an isomorphism. Recall that for any closed orientable 4-manifold $V$ and any $v\in H^1(V,\F2)$, we have $v^4=0$. Thus, $\omega^4=0$, and as $\gamma_0$ is an isomorphism, 
    \begin{equation}
    \label{eqn:omega3=0}
        \omega^3=0.  
    \end{equation}
    \item $\gamma_2: H_3(M/\sigma, \F2)\xrightarrow{\cap \omega} H_2(M/\sigma, \F2)$
    is a monomorphism, hence
    \begin{equation}
    \omega^2\ne 0.
    \end{equation}
    \item We have an exact sequence:
    {\footnotesize
    \begin{equation} \label{eqn:MiddleExactSeq}
        0\to H_3(M/\sigma, \F2)\xrightarrow{\gamma_2=\cap \omega} H_2(M/\sigma, \F2)\xrightarrow{\alpha_2=\pi^*} H_2(M, \F2) \xrightarrow{\beta_2=\pi_*}  H_2(M/\sigma, \F2)\xrightarrow{\gamma_1=\cap \omega}H_1(M/\sigma, \F2)\to 0.
    \end{equation}
    }
\end{itemize}
In its turn, from the exactness of (\ref{eqn:MiddleExactSeq}) it follows that
$1+\rk \alpha_2=b_2(M/\sigma,\F2)=\rk\beta_2+1$ and $\rk\alpha_2+\rk\beta_2=b_2(M,\F2)$.
Note that by assumption $H^*(M, \ZZ)$ is 2-torsion-free, hence $b_2(M, \F2)$ equals to the usual $b_2(M)$. The relation \eqref{eqn:RankAlpha2}, and the second equality in 
\eqref{eqn:I2Dim}, are proven.


Next, since $(\alpha_2\circ\beta_2)I_2=(1+\sigma_*)I_2=0$, we get
$\beta_2(I_2)\subset \ker(\alpha_2)=\Im \gamma_2$ and $\dim\beta_2(I_2)\le \dim\Im\gamma_2=1$. Thus, due to $\ker \beta_2=\Im \alpha_2\subset I_2$,
to prove $\dim I_2=\dim \Im \alpha_2+1$ it is sufficient to check that 
$\dim \beta_2 (I_2)\ge 1$.


Now, from $\omega^3=0$ in \eqref{eqn:omega3=0} and the exactness of \eqref{eqn:MiddleExactSeq}, it follows that there exists $\xi\in H_2(M,\F2)$ with $\beta_2(\xi)= D\omega^2$
(where $D$ stands for the Poincar\'e duality). 
We have $\xi\in I_2$, since $(1+\sigma_*)\xi=
(\alpha_2\circ\beta_2)(\xi)=
\alpha_2(D\omega^2)=
\alpha_2(\omega\cap D\omega)=(\alpha_2\circ \gamma_2)(D\omega)=0$ by exactness of the Smith--Gysin sequence.
As $\omega^2\neq 0$, $\dim \beta_2(I_2)\ge 1$, which concludes the proof of $\dim I_2=\dim \Im \alpha_2+1$.
\end{proof}

\begin{rmk}
\label{rmk:RankOf+-lattices}
By the Lefschetz trace formula for fixed points, we have
\begin{equation}
    \rk H^2(M, \ZZ)^{\sigma-}- \rk H^2(M, \ZZ)^{\sigma}=2.
\end{equation}
Hence $b_2(M)$ is an even number and
\begin{equation}\label{A9}
    \rk H^2(M, \ZZ)^{\sigma-}=\frac{1}{2} b_2(M)+1; \quad\quad 
        \rk H^2(M, \ZZ)^{\sigma}=\frac{1}{2} b_2(M)-1.
\end{equation}
By Lemma \ref{InvClasses} it follows that
\begin{equation}
\dim I_2=\rk H^2(M, \ZZ)^{\sigma-}.    
\end{equation} 
This implies that the pull-back homomorphism
establishes a lattice isomorphism: 
\begin{equation}
   \pi^*\colon  H^2(M/\sigma , \ZZ)(2) \xrightarrow{\cong } H^2(M, \ZZ)^{\sigma}.
\end{equation}
\end{rmk}
\begin{theorem}
\label{thm:FixedPointTheorem}
    Let $S$ be a compact complex surface with $H^1(S, \F2)=0$. Let $\sigma$ be an involution of $S$ without fixed point that satisfies one of the following conditions
    \begin{enumerate}
        \item[(i)]  $\sigma$ is a holomorphic involution, or
        \item[(ii)]  $\sigma$ is anti-holomorphic and $b_2(S)\neq 2$.
    \end{enumerate}
    Then $\sigma$ acts on $H_2(S, \F2)$ non-trivially.
\end{theorem}
\begin{proof} 
Assume for contradiction that $\sigma$ acts on $H_2(S, \F2)$ trivially, that is, $$\dim I_2=b_2(S).$$ 
As a holomorphic or anti-holomorphic involution preserves the natural orientation of a complex surface, we can apply Lemma \ref{InvClasses}, and obtain that 
\begin{equation}
    \dim I_2=\frac{1}{2}b_2(S)+1.
\end{equation}
Thus the only possibility is when $b_2(S)=2$. Hence, case (ii) is proven. 

For case (i), i.e.~$\sigma$ is holomorphic, by the Lefschetz trace formula, $\sigma$ acts on $H^2(S, \QQ)$ by $-\id$; see (\ref{A9}) in \Cref{rmk:RankOf+-lattices}. As the canonical class is preserved by $\sigma$, this implies that $K_S$ is torsion. In particular, $S$ is a minimal surface of Kodaira dimension 0. By looking at the Enriques--Kodaira classification, there is no such type of surfaces with $b_1(S)=0$ and $b_2(S)=2$. Case (i) is proven. 
\end{proof}

\begin{cor}
\label{cor:FixedPointFree-NonMaximal}
Let the notations and assumptions be as in \Cref{thm:FixedPointTheorem}. For a given positive integer $n$, let $S^{[n]}$ be the $n$th Hilbert scheme  of points on $S$, and let $\sigma^{[n]}$ be the naturally induced (holomorphic or anti-holomorphic) involution on $S^{[n]}$. Then $\sigma^{[n]}$ acts non-trivially on $H^2(S^{[n]}, \F2)$. In particular, $\sigma^{[n]}$ is not a maximal involution.
\end{cor}
\begin{proof}
The case where $n=1$ is exactly \Cref{thm:FixedPointTheorem}. Assume $n\geq 2$ in the sequel. 
The assumption $H^1(S, \F2)=0$ implies that $H^1(S, \ZZ)=0$ and $H^2(S, \ZZ)$ is 2-torsion-free.
We have the following isomorphism (see for example \cite[P.768]{Beauville}), where $(-)_{\operatorname{tf}}$ stands for $(-)/\operatorname{tors}$.
\begin{equation}
\label{eqn:IsomH2}
    H^2(S^{[n]}, \ZZ)_{\operatorname{tf}} \cong H^2(S, \ZZ)_{\operatorname{tf}} \oplus \ZZ\cdot \delta,
\end{equation}
where $\delta$ is half of the class of the exceptional divisor in $S^{[n]}$. In the isomorphism \eqref{eqn:IsomH2}, the injection $i\colon H^2(S, \ZZ)_{\operatorname{tf}} \to H^2(S^{[n]}, \ZZ)_{\operatorname{tf}}$ is induced by the incidence subscheme $S^{[1, n]}:=\{(x, \xi)\in S\times S^{[n]}~|~ x\in \operatorname{supp}(\xi)\}$. Therefore $i$ is equivariant with respect to the action of $\sigma$ and $\sigma^{[n]}$.

Since $H^1(S, \F2)=0$, the cohomology $H^*(S, \ZZ)$ is 2-torsion-free.  By \cite[Theorem 3.1 and the remark that follows]{TotaroHilbn}, $H^*(S^{[n]}, \ZZ)$ is also 2-torsion-free. We obtain from \eqref{eqn:IsomH2} a $(\sigma^{[n]}, \sigma)$-equivariant isomorphism:
\begin{equation}
        H^2(S^{[n]}, \F2) \cong H^2(S, \F2)\oplus \F2\cdot \delta.
\end{equation}
Since the action of $\sigma$ on $H^2(S, \F2)$ is non-trivial, the action of $\sigma^{[n]}$ on $H^2(S^{[n]}, \F2)$  is non-trivial.
The non-maximality follows from \Cref{prop:MaximalityViaCohomology}.
\end{proof}

\begin{rmk}
    Let us give an alternative geometric proof for the non-triviality of the action of $\sigma^{[n]}$ on $H^2(S^{[n]}, \F2)$ without using \eqref{eqn:IsomH2}. Choose an element $\alpha\in H^2(S, \F2)$ with $\sigma_*(\alpha)\neq \alpha$. By perfectness of the intersection pairing, there exists $\beta\in H^2(S, \F2)$ with $\alpha\cup\beta=0$ and $\sigma_*(\alpha)\cup\beta\neq 0$. 
Let $S^{[1, n]}:=\{(x, \xi)\in S\times S^{[n]}~|~ x\in \operatorname{supp}(\xi)\}$ as before, and let 
$S_0^{[1, n]}:=\{(x, \xi)\in S\times S^{[n]}~|~ \{x\}\subset\xi \text{ and } \operatorname{supp}(\mathcal{O}_\xi/\mathcal{O}_x) \text{ is one point}\}$. There is a natural residual-point morphism $r\colon S_0^{[1,n]}\to S$.
Consider
\begin{align*}
\tilde{\alpha}&:= \1_{-(n-1)}(\alpha)=(S^{[1,n]})_*(\alpha) \in H^2(S^{[n]}, \F2);\\
\quad\quad \tilde\beta&:=\p_{-1}(\pt)^{n-1}(\alpha)= (r^*(\pt))_*(\beta)\in H^{4n-2}(S^{[n]}, \F2).    
\end{align*}
Then the intersection pairings $(\tilde\alpha\cdot \tilde{\beta})=(\alpha\cdot \beta)=0$ and $(\sigma^{[n]}_*\tilde\alpha\cdot \tilde{\beta})=(\sigma_*\alpha\cdot \beta)\neq0$. In particular, $\tilde\alpha$ is not preserved by $\sigma^{[n]}$.
\end{rmk}

\begin{rmk}
\label{rmk:Remaining cases}
The only cases that are not covered by \Cref{thm:FixedPointTheorem} are fixed-point-free anti-holomorphic involutions on compact complex surfaces with $H^1(S,\F2)=0$ and $b_2(S)=2$. Thanks to the Enriques--Kodaira classification, such surfaces can only be (smooth) quadrics, fake quadrics\footnote{A fake quadric in this paper is always assumed to be of general type. See the precise definition in \Cref{sec:FakeQuadrics}.}, Hirzebruch surfaces, 
and blown-ups of fake projective planes at one point. The last case
can be easily excluded by noticing that any real structure on the blown-up of a fake projective plane must globally preserve the exceptional $(-1)$-curve $E$, hence $(E\cdot \sigma_*(E))$ is an odd number, but for a real structure $\sigma$ on a smooth projective surface without real points,  $(D\cdot \sigma_*(D))$ must be an even number for any divisor $D$. 
We will treat the remaining cases, namely, quadrics, fake quadrics and Hirzebruch surfaces in the rest of this section, by using the so-called Kalinin spectral sequence.
\end{rmk}

\subsection{Kalinin spectral sequence}

To treat the case of quadrics, Hirzebruch surfaces and fake quadrics, we apply the so-called {\it Kalinin spectral sequence}. This spectral sequence can be deduced from the exact Smith sequence or can be seen as a kind of the stable part of the Borel--Serre spectral sequence (see \cite{RealEnriques}, for example). Contrary to most traditional spectral sequences, Kalinin spectral sequence is $\ZZ$-graded on each page. More precisely, for a manifold $M$ equipped with an involution $\sigma$, the Kalinin spectral sequence is built as follows:
\begin{itemize}
    \item Page $E^0$ is the chain complex of $M$ with the usual boundary operator as the differential $d_0$.
    \item In page $E^1$, the terms are the usual homology groups $H_*(M, \F2)$ and the differential $d_1$ is the "averaging" operator:
    \begin{align*}
              d_1\colon   H_r(M,\F2) &\to  H_r(M,\F2).\\
 x_r &\mapsto  x_r+\sigma_* x_r
    \end{align*}
    \item In page $E^2$, the terms are $H_*(M,\F2)^\sigma/\Im(1+\sigma_*)$, and the differential  
    \begin{equation}
    d_2 : H_r(M,\F2)^\sigma/\Im(1+\sigma_*)\to H_{r+1}(M,\F2)^\sigma/\Im(1+\sigma_*)  
    \end{equation}
    is described as follows.
Starting from $x_r\in H_r(M,\F2)^\sigma/\Im(1+\sigma_*)$, we select
an $r$-dimensional cycle $\eta_r$ representing $x_r$.
Then one can choose a chain $\eta_{r+1}$ such that
$\partial \eta_{r+1}=\eta_r+\sigma_*(\eta_r)$. We define $d_2(x_r)$ to be the class of 
$\eta_{r+1}+\sigma_*\eta_{r+1}$. It is straightforward to check that $d_2$ is well-defined.
\item For the differential $d_3$ on the $E^3$-page, we only give its description under the extra assumption that $H_{\mathrm{odd}}(M,\F2)=0$ (which is satisfied in the applications). In such a situation we have trivially $d_2=0$. For any even integer $r$,
\begin{equation}
d_3\colon H_r(M,\F2)^{\sigma}/\Im(1+\sigma_*)\to H_{r+2}(M,\F2)^{\sigma}/\Im(1+\sigma_*)
\end{equation} 
has the following chain description (see \cite[P.9]{RealEnriques}):
for $x_r \in H_r(M,\F2)^\sigma/\Im(1+\sigma_*)$, we choose an $r$-dimensional cycle $\eta_r$ representing $x_r$. By assumption we can find an $(r+1)$-dimensional chain $\eta_{r+1}$ and an $(r+2)$-dimensional chain $\eta_{r+2}$, such that $\partial \eta_{r+2}=\eta_{r+1}+\sigma_*\eta_{r+1}$ and $\partial \eta_{r+1}= \eta_{r}+\sigma_*\eta_{r}$, then we
define $d_3(x_r)$ as the class of $\eta_{r+2}+\sigma_*\eta_{r+2}$. One can check that $d_3$ is well-defined in this case.
\end{itemize}
As a straightforward application, one gets the following
obstruction to maximality that we will exploit in this section: 
\begin{center}
    {\it If one of the differentials $d_r$, $r\ge 1$, is non-zero in the Kalinin spectral sequence, then $\sigma$ is not maximal.}
\end{center}

To illustrate the use of this obstruction as well as the computation of differentials in Kalinin spectral sequence, we give two examples.
\begin{ex}[Quadrics]
\label{example:Quadrics}
Recall that up to conjugation, $\PP^1$ admits only two real structures (equivalently real forms): the standard one $z\mapsto \overline{z}$ giving rise to $\PP^1_{\RR}$ as real form, and the \textit{antipode} $z\mapsto -\frac{1}{\overline{z}}$ giving rise to the conic without real points as real form.
Up to isomorphism, there are only two real structures on $\PP^1\times \PP^1$ with empty real loci (\cite{Comessatti-RationalSurfaces}). Namely, 
\begin{equation}
   (z_1, z_2) \mapsto (\overline{z_1}, -\frac{1}{\overline{z_2}}) \text{ and } (z_1, z_2)\mapsto (-\frac{1}{\overline{z_1}}, -\frac{1}{\overline{z_2}}).
\end{equation}

In either case, the anti-holomorphic involution acts trivially on $H_2(\PP^1\times \PP^1,\F2)$, 
while the "line generator" $\PP^1\times \{\pt\}$ is not invariant.
This leads to the conclusion that, in both cases, 
the class $\PP^1\times [\pt]\in H_2(\PP^1\times \PP^1,\F2)$ 
of the latter line generator
is sent by the differential in the $E^3$-page of the Kalinin spectral sequence 
to the fundamental class of the surface: $$d_3(\PP^1\times [\pt])=[\PP^1\times \PP^1].$$
Indeed, following the chain construction recalled above, we find\\
-\quad $\eta_2=\PP^1\times\{z\}$, where $z$ is an arbitrary point in $\PP^1$,\\
-\quad $\eta_3= \PP^1\times\operatorname{arc}$ where $\operatorname{arc}$ is a semi-big-circle connecting $z$ with its antipode $-\frac{1}{\overline{z}}$,\\
-\quad $\eta_3+\sigma_*\eta_3=\PP^1\times\operatorname{circle}$ where $\operatorname{circle}$ is a big-circle through  $z$ and $-\frac{1}{\overline{z}}$, and finally\\
-\quad $\eta_4=\PP^1\times \operatorname{hemisphere}$, hence $d_3(\PP^1\times [\pt])=[\eta_4+\sigma_*\eta_4]=[\PP^1\times \PP^1]$.
\end{ex}

\Cref{example:Quadrics} can be generalized to Hirzebruch surfaces as follows:

\begin{ex}[Hirzebruch surfaces]
\label{example:Hirzebruch}
For each integer $e>0$, the $e$-th Hirzebruch surface is a complex surface $S$ isomorphic to the ruled surface obtained as the projectivization of the rank-2 vector bundle 
$\mathcal{O}\oplus \mathcal{O}(-e)$ over a base curve $C\cong \PP^1$. Let $\pi\colon S\to C$ be the $\PP^1$-bundle projection. $S$ has a unique irreducible curve $\Sigma$ with self-intersection $-e$, which is 
the section of $\pi$ defined by the first summand of $\mathcal{O}\oplus \mathcal{O}(-e)$. Denote by $F$ the fiber class of $\pi$. Then we have $(F^2)=0$, $(F\cdot \Sigma)=1$, and $(\Sigma^2)=-e$.

Let $\sigma$ be a real structure on $S$ without real points. Since $\sigma$ preserves $\Sigma$ (reversing its orientation) and the intersection form, $\sigma$ acts on $H^2(S, \ZZ)$ as $-\id$. In particular, $\sigma_*(F)=-F$.  Therefore $\sigma$ preserves the $\PP^1$-bundle structure, and induces a real structure $\tau$ on the base curve $C$. As $\pi$ induces an equivariant isomorphism between $\Sigma$ and $C$, $\tau$ has no fixed point, i.e.~it is the antipode on $C$. Moreover, as $\sigma$ has no fixed points, the intersection form must be even, hence $e$ must be even. In fact, when $e>0$ is even, by Comessatti \cite{Comessatti-RationalSurfaces} and Iskovskih \cite{Iskovskih-MinimalModelRationalSurface},
there is a unique real structure on $S$ without real points, up to isomorphism.

The image of the fiber class $F$ under the differential $d_3$ in the Kalinin spectral sequence is the fundamental class of $S$:
\begin{equation}
    d_3(F)=[S]
\end{equation}
Indeed, for any $z\in C$, let $F_z\cong \PP^1$ denote the fiber of $\pi$ over $z$. As in \Cref{example:Quadrics}, let us choose an arbitrary $z\in C$ and represent $F$ by a fiber $\eta_2:=F_z$. 
Then $\eta_3:=\pi^{-1}(\operatorname{arc})=\bigcup_{z\in \operatorname{arc}}F_z$ satisfies  $\partial \eta_3=\eta_2+\sigma_*\eta_2$, where $\operatorname{arc}$ is a semi-big-circle connecting $z$ with its antipode $\tau(z)$. Hence $\eta_3+\sigma_*\eta_3=\pi^{-1}({\operatorname{circle}})$, where $\operatorname{circle}$ is a big-circle through  $z$ and $\tau(z)$. 
Finally, $\eta_4=\pi^{-1}(\operatorname{hemisphere})$ satisfies $\partial \eta_4=\eta_3+\sigma_*\eta_3$ and we have $d_3(F)=[\eta_4+\sigma_*\eta_4]=[S]$.

\end{ex}
To summarize, we have proved the following statement.

\begin{prop}\label{gen-prop} 
Let $S$ be  $\PP^1\times \PP^1$ or a Hirzebruch surface. For any real structure on $S$ without real points, let $F \in H_2(S,\F2)$ be the fiber-class of a ruling over the conic without real point if $S=\PP^1\times \PP^1$, and 
that of the unique ruling on $S$, if $S$ is a Hirzebruch surface. Then the image of $F$ under the differential $d_3$ of the third page of the Kalinin spectral sequence is the fundamental class of $S$. In particular, $d_3\neq 0$.
\end{prop}

\subsection{Quadrics and Hirzebruch surfaces}
Following \Cref{example:Quadrics} and \Cref{example:Hirzebruch}, we prove the following results.
\begin{theorem}\label{gen-theorem}
Let $S$ be the complex surface $\PP^1\times \PP^1$ equipped with a real structure $\sigma$ without real points. Then for any positive integer $n$ the $n$th Hilbert scheme $S^{[n]}$, equipped with the natural real structure, is not maximal. 
\end{theorem}

\begin{proof}
We prove that at least one of the differentials $d_1, d_2, d_3$ in Kalinin's spectral sequence for the natural involution on $S^{[n]}$ is not zero. For that, we assume that $d_1$ and $d_2$ are zero and check that $d_3\ne 0$. 

Since $d_1$ and $d_2$ are zero, $d_3$ is a linear map from $H_*(S^{[n]},\F2)$ to $H_{*+2}(S^{[n]},\F2)$.
We follow the notation in \Cref{example:Quadrics}. We can assume the real structure on $S=\PP^1\times \PP^1$ is given by $(\tau', \tau)$ with $\tau$ the antipode real structure on $\PP^1$. Choose a point $t\in \PP^1$, and denote by $x$ the class in $H_{2n}(S^{[n]}, \F2)$ represented by the cycle $\Sym^n(\PP^1\times \{t\})$. Our goal is to check that $d_3(x)\ne 0$.

Using the chain construction of $d_3$ and proceeding as in the proof of \Cref{gen-prop}, we observe that $d_3(x)$ is represented by the following cycle of dimension $2n+2$:
\begin{equation}
    M:=\bigcup_{t\in \PP^1} \Sym^n(\PP^1\times \{t\}).
\end{equation}
In order to show that $[M]\neq 0$, we construct a complementary cycle of dimension $2n-2$. Pick
$n-1$ distinct points $p_1,\dots, p_{n-1}$ in $\PP^1$ and a point
$q=(q_1,q_2)\in \PP^1\times \PP^1$ with $q_1\ne p_1,\dots,p_{n-1}$. We let $y\in H_{2n-2}(S^{[n]}, \F2)$
be the class represented by the cycle $N$ formed by $w\in S^{[n]}$ with
$\operatorname{supp}(w)$ consisting of a fixed point $q$ and variable points $(p_1,t_1), \dots (p_{n-1},t_{n-1})$ with $t_1, \dots, t_{n-1}\in \PP^1$. Clearly, $N\cap M$ consists of the unique length-$n$ reduced subscheme
with support in $(q_1,q_2), (p_1,q_2), \dots (p_{n-1}, q_2)$ and their intersection is transversal. Hence, $([N]\cdot [M])=1\in\F2$,
which implies $d_3(x)=[M]\ne 0$.
\end{proof}

\begin{cor}
\label{cor:HirzebruchHilb}
    Let $S$ be a Hirzebruch surface. Let $\sigma$ be a real structure on $S$ without real points. Then for any $n\geq 1$, the Hilbert scheme $S^{[n]}$ equipped with the natural real structure is not maximal.
\end{cor}
\begin{proof}
    We use the notation in \Cref{example:Hirzebruch}. Let $-e<0$ be the self-intersection number of the exceptional section, which must be an even number since $\sigma$ has no fixed points. By \cite{Comessatti-RationalSurfaces} (see also \cite{Silhol-Surface-LNM} and \cite[2.5.2]{Degtyarev-Kharlamov-Crelle}), the real structure is actually unique up to isomorphism. By Degtyarev--Kharlamov \cite{Degtyarev-Kharlamov-Crelle}, real rational surfaces are quasi-simple, hence the $\RR$-surface $(S, \sigma)$ is \textit{real deformation equivalent} to the $\RR$-surface $(\PP^1\times \PP^1, \tau\times \tau)$ where $\tau$ is the antipode real strucutre on $\PP^1$; see \cite[Section 4.2, Case 3]{Degtyarev-Kharlamov-Crelle} for the explicit construction of the deformation. Therefore, the Hilbert scheme $(S^{[n]}, \sigma^{[n]})$ is real deformation equivalent to $((\PP^1\times \PP^1)^{[n]}, (\tau\times \tau)^{[n]})$. Since the latter is not maximal by \Cref{gen-theorem}, while maximality being clearly a real deformation invariant property, we can conclude the non-maximality of $\sigma^{[n]}$.
\end{proof}

\subsection{Fake quadrics}\label{sec:FakeQuadrics}
Let us understand by a {\it fake quadric} a minimal smooth projective surface $S$ of general type such that $q(S)=p_g(S)=0$, $b_2(S)=2$ and $K_S^2=8$. In particular, each fake quadric has the same Betti and Hodge numbers as smooth quadrics in $\PP^3$.
For a fake quadric $S$, its N\'eron--Severi lattice $\mathrm{N}^1(S):=H^2(S,\ZZ)_{\operatorname{tf}}$ is a unimodular indefinite lattice of rank 2, hence is isometric to either $U$ or $\langle1\rangle+\langle-1\rangle$. A fake quadric $S$ is called the {\it odd type} (resp. {\it even type})
if $H^2(S,\ZZ)_{\operatorname{tf}}$ as a lattice is odd (resp. even), i.e.~isometric to  $\langle1\rangle+\langle-1\rangle$ (resp. to $U$).

\begin{theorem}
\label{thm:FakeQuadrics}
Fake quadrics of odd type with $H^1(S, \F2)=0$ admit neither
holomorphic or anti-holomorphic involution without fixed points.
Fake quadrics of even type with $H^1(S, \F2)=0$ do not admit any holomorphic involution without fixed points.
\end{theorem}

\begin{proof} 
Let $\sigma$ be an involution without fixed points. From the Lefschetz trace formula, it follows that $\sigma$ acts on $H^2(S,\ZZ)$ by $-\id$. On the other hand, $\sigma$ has no fixed point implies that $(D\cdot \sigma_*D)=0\mod 2$ for every $D\in H^2(S,\ZZ)$. Both together show that $S$ is of even type. The second statement follows from the observation that holomorphic involutions preserve the class $\frac{1}{2} K_S
\in H^2(X,\ZZ)$ and that its reduction $\mod 2$ is not zero.
\end{proof}

Thanks to \Cref{thm:FakeQuadrics}, we can concentrate in real structures on fake quadrics of even type. Recall that in this case, $H^2(S,\ZZ)_{\operatorname{tf}}\cong U$. Since $K_S$, as well as the first Chern class $c_1(S)$, is mapped by the coefficient homomorphism to the Wu class, we have that 
$(D\cdot K_S)\equiv (D^2)\mod 2$ 
for any $D\in H^2(S,\ZZ)$. 
Hence, $K_S$ has even intersection number with any 
$D\in H^2(S,\ZZ)$.
Thus, $\frac{1}{2}K_S$ is an element $H^2(S,\ZZ)_{\operatorname{tf}}$, whose self-intersection number is 2 (as $K_S^2=8$).
Therefore there exists a basis
$\{H,F\}$ of $H^2(S,\ZZ)$ such that $(H^2)=(F^2)=0, (F\cdot H)=1$ and $K_S=2H+2F$.
\begin{lemma}\label{fake-lemma} 
Let $S$ be a fake quadric of even type and let $\sigma:S\to S$ be an anti-holomorphic involution without fixed point. Then, there exist on $S$ two linear pencils of curves, $\{A_t\}_{t\in \PP^1}$ and $\{B_t\}_{t\in \PP^1}$, satisfying the following properties: 
their divisor classes have the same$\mod 2$-reduction as $H$ and $F$, in particular, $(A_u\cdot B_v)=1\mod 2$ for every $u,v\in \PP^1$;
both pencils are $\sigma$-invariant, and for at least one of them, say for $\{A_t\}$, it holds that $\sigma(A_t)=A_{\tau(t)}$ where $\tau : \PP^1\to \PP^1$ is a real structure without real points.
\end{lemma}

\begin{proof} 
We pick inside the interior of the ample cone of $S$ two ample divisors, 
$A'=xH+yF$ with 
odd $x$ and even $y$, and $B'=uH+vF$ with even $u$
and odd $v$. Then, the divisors $A=mA'$ and $B=mB'$ are very ample for each $m$
sufficiently big, which we take odd to preserve the property that $(A\cdot B)=m^2(A'\cdot B')=m^2(xv+yu)$ is odd. 
Since $\Pic S=H_2(S,\ZZ)$ and the action of $\sigma$ on it is a multiplication by $-1$, we can lift $\sigma: S\to S$ up to anti-automorphisms $c_A$, $c_B$
of the line bundles $L_A, L_B$ defined by the divisor classes $A, B$. Then, the transformation $f\mapsto c_A\circ f\circ \sigma$ defines a real structure in the projectivization, $\vert A\vert$, of the spaces of sections of $L_A$; similarly for $\vert B\vert$. Due to Lemma \ref{InvClasses}, at least in one of these spaces, the real structure constructed is without real points.
Indeed, if $|A|$ has a real point, then $A$ is linearly equivalent to a divisor that is preserved by $\sigma$, hence $A\in \Im(\pi^*)$; similarly for $|B|$. However, Lemma \ref{InvClasses} implies that $\dim \Im(\pi^*:H_2(S/\sigma, \F2)\to H_2(S, \F2))=\dim H_2(S, \F2)^{\sigma}-1=1$. Since the classes of $A$ and $B$ are $\F2$-linearly independent, it is impossible that $|A|$ and $|B|$ both admit  real points. Finally, it remains to pick a real line in each of these two projective spaces of sections.
\end{proof}

\begin{theorem}\label{fake-quadric-theorem}
If $\sigma: S\to S $ is an anti-holomorphic involution without fixed point on a fake quadric $S$ of even type, then for any positive integer $n$ the $n$-th Hilbert scheme $S^{[n]}$ is not maximal with respect to the induced involution $\sigma^{[n]}$.
\end{theorem}

\begin{proof} The proof follows the same arguments as the proof of Theorem \ref{gen-theorem} given above.

Here, we apply Lemma \ref{fake-lemma} and consider the element $x_{2n}$ of $H_{2n}(S^{[n]}, \F2)$ represented by the cycle $\Sym^{n} A_{t}$ with a chosen point $t\in \PP^1$. Proceeding as in the proof of Proposition \ref{gen-prop} we observe that $d_3(x_{2n})$ is represented by the following cycle of dimension $2n+2$:
\begin{equation}
    M:=\bigcup_{t\in \PP^1} \Sym^n A_t.
\end{equation}
Our goal is to show $d_3(x_{2n})=[M]\neq 0$. 

To construct a complementary cycle $N$ of dimension $2n-2$, we pick a generic point $q\in A_{t}$ and choose $n-1$ distinct curves $B_{t_i}$, $t_i\in \PP^1$, $i=1,\dots, n-1$ not passing through $q$. Then we define $N$ to be the cycle formed
by $w\in S^{[n]}$ with $\operatorname{supp}(w)$ consisting from the chosen point $q$ and variable points $p'_1\in B_{t_1}, \dots p'_{n-1}\in B_{t_{n-1}}$. As it follows from Proposition \ref{gen-prop}, $N\cap M$ consists of an odd number of points (equal to $(A_u\cdot B_v)^{n-1}$) and their intersection is transversal. Hence, $([N]\cdot [M])=1\in\F2$,
which implies $[M]\ne 0$.
    
\end{proof}

\subsection{Proof of {\Cref{thm:NonMaxHilbSurfWithoutFixPoint}}}\label{sec:proof_of_NonMaxHilbSurfWithoutFixPoint}
We are now ready to prove the main theorem of this section, which is obtained as a recollection of the previous results.
\begin{proof}[Proof of \Cref{thm:NonMaxHilbSurfWithoutFixPoint}]
Assume \(S\) is a compact complex surface with \(H^1(S,\mathbb{F}_2)=0\) and let \(\sigma\) be a holomorphic or anti-holomorphic involution on \(S\) without fixed points. By \Cref{thm:FixedPointTheorem} we know that if \(\sigma\) is holomorphic, or \(\sigma\) is anti-holomorphic and \(b_2(S)\not=2\), then \(\sigma \) acts non-trivially on \(H_2(S,\mathbb{F}_2)\). In particular, by \Cref{cor:FixedPointFree-NonMaximal} the induced involution on its Hilbert scheme of points is not maximal. The case where \(\sigma\) is anti-holomorphic and \(b_2(S)=2\) remains to be treated. From \Cref{rmk:Remaining cases}, namely the Enriques--Kodaira classification with the fact that the blow-up at a point of a fake projective plane does not admit fixed-point-free real structures, it follows that we are left to consider only smooth quadrics, fake quadrics and Hirzebruch surfaces. An anti-holomorphic involution without fixed points (i.e.~real structure without real points) on a quadric induces a non-maximal involution on the Hilbert schemes of points by virtue of \Cref{gen-theorem}, which, by real deformation, leads to the same result for Hirzebruch surfaces as stated in \Cref{cor:HirzebruchHilb}. It remains to treat the case of fake quadrics. Fake quadrics 
of odd type do not admit fixed-point-free anti-holomorphic involutions by \Cref{thm:FakeQuadrics}. Fake quadrics of even type might admit anti-holomorphic involutions of without fixed points, but they induce involutions that are not maximal on the Hilbert schemes of points by \Cref{fake-quadric-theorem}.
\end{proof}

\section{Natural involutions on Hilbert schemes of points on surfaces}
\label{sec:Natural}
In this section, $S$ is a smooth projective complex surface. 
We assume that 
\begin{equation}
	\label{eqn:H1F2=0}
	H^1(S, \F2)=0
\end{equation}

\begin{rmk}
	\label{rmk:TorsionFree}
	Note that the condition (\ref{eqn:H1F2=0}) is equivalent to requiring that $b_1(S)=0$ and $H^*(S,\ZZ)$ is 2-torsion-free. By G\"ottsche \cite{Goettsche} and 
    Totaro \cite{TotaroHilbn}, this also implies that $S^{[n]}$ has vanishing odd Betti numbers and 2-torsion-free integral cohomology. \textit{In the sequel, since all cohomology groups are 2-torsion-free and torsion elements of odd order are irrelevant, we ignore the torsion in cohomology groups, and, when it does not lead to a confusion, we drop $\,/\tors $ from the notation and denote $H^*(-, \ZZ)/\tors$ by
    $H^*(-, \ZZ)$.}
\end{rmk}
A holomorphic (resp.~anti-holomorphic) involution $\sigma$ on $S$ naturally induces a holomorphic (resp.~anti-holomorphic) involution on $S^{[n]}$.
The goal of this section is to relate the maximality of this induced involution on $S^{[n]}$ to conditions on the pair $(S, \sigma)$.

\subsection{Natural anti-holomorphic involutions on Hilbert schemes}
The following theorem completes Fu \cite[Theorem 8.1]{FuMaximalReal} to an \textit{if-and-only-if} result and generalizes Kharlamov--R\u asdeaconu \cite[Theorem 1.1, Theorem 1.2]{Kharlamov-Rasdeaconu-HilbertSquare} to all dimensions. Even for $n=2$, our proof is different (\cite{Kharlamov-Rasdeaconu-HilbertSquare} studies in much more details the fixed locus in the $n=2$ case).
\begin{theorem}
	\label{thm:NaturalAntiHoloInv}
	Let $n\geq 2$. Let $S$ be a smooth projective $\RR$-surface. Assume that $H^1(S, \F2)=0$.
	Then the punctual Hilbert scheme $S^{[n]}$ equipped with the natural real structure is maximal if and only if $S$ is maximal with connected real locus, or equivalently, maximal with the real structure 
    acting as $-\id$ on $H^2(S, \ZZ)$.
\end{theorem}
\begin{proof}
First, by \cite[Proposition 4.4]{FuMaximalReal},
for a maximal smooth projective surface $S$ with $H^1(S, \F2)=0$, the connectedness of $S(\RR)$ is equivalent to $H^2(S,\ZZ)^{\sigma}=0$, which in its turn is equivalent to the condition that $\sigma$ acts as $-\id$ on $H^2(S,\ZZ)$. 
	The "if" part of the theorem is proven in Fu \cite[Theorem 8.1]{FuMaximalReal}. For the "only if" part, let $\sigma$ denote both the associated anti-holomorphic involutions on $S$ and on $S^{[n]}$.  Let $G$ be the cyclic group of order 2 generated by $\sigma$.
	
	Suppose that the involution is maximal on $S^{[n]}$.  Let us start by showing that $S$ is maximal. Thanks to \Cref{thm:NonMaxHilbSurfWithoutFixPoint}, we may assume that $S(\RR)\neq \emptyset$. By \Cref{prop:MaximalityViaCohomology}, in order to prove the maximality of $(S, \sigma)$, 
    it is sufficient to show that for any $i$, the natural map $H^i_G(S,\F2)\to H^i(S, \F2)$ is surjective. For $i=0,1, 3$, it is clear from the assumption;  for $i=4$, it follows from the existence of fixed/real points. For $i=2$, note that, in accord with \Cref{rmk:IntegralBasisH2H4}, the universal subscheme $S^{[1,n]}\subset S\times S^{[n]}$ induces a $G$-equivariant embedding with 2-torsion-free quotient:
	\begin{equation}
		\1_{-(n-1)}\colon H^2(S, \ZZ)\hookrightarrow H^2(S^{[n]}, \ZZ).
	\end{equation}
	The maximality of $S^{[n]}$ implies that the Comessatti characteristic of $H^2(S^{[n]}, \ZZ)$ is zero. By \Cref{lemma:ComessattiPrimitiveSubModule}, the Comessatti characteristic of $H^2(S, \ZZ)$ is also zero. By \Cref{lemma:EquivalentConditionsForSplitting}, $G$ acts trivially on $H^2(S, \F2)$. Now the Leray--Serre spectral sequence \eqref{eqn:LeraySerreSS}, together with the vanishing of $H^1(S, \F2)$, yields an exact sequence
	\begin{equation}
		H^2_G(S, \F2)\to H^2(S, \F2)^G\to H^3(G, \F2)\to H^3_G(S, \F2)
	\end{equation}
	The last map is split injective since a real point on $S$  provides a section of $S_G\to \mathsf BG$. Hence $H^2_G(S, \F2)\to H^2(S, \F2)^G=H^2(S, \F2)$ is surjective. We finished proving that $H^*_G(S, \F2)\to H^*(S, \F2)$ is surjective, hence $S$ is maximal.

	For the connectedness of the real locus, let us assume that $S(\RR)$ is not connected, and show that $S^{[n]}$ is not maximal.  By \cite[Proposition 4.4]{FuMaximalReal}, the non-connectedness of $S(\RR)$ implies that $H^2(S, \ZZ)^{\sigma}\neq 0$. 
	Choose a primitive element $\alpha\in H^2(S, \ZZ)^{\sigma}$, and  put
	\begin{equation}
		v:= \1_{-(n-2)}\p_{-2}(\alpha)|0\rangle.
	\end{equation}
	By definition of the Nakajima and Li--Qin--Wang operators (see \Cref{subsec:Nakajima-Li-Qin-Wang-Operators}), $v$ is the image of $\alpha$ via the following composition of correspondences:
	\begin{equation}
		H^2(S,\ZZ)\xrightarrow{[E]_*} H^4(S^{[2]}, \ZZ) \xrightarrow{[S^{[2, n]}]_*} H^4(S^{[n]}, \ZZ),
	\end{equation}
	where $E$ is the exceptional divisor in $S^{[2]}$ with $E\to S$ the natural $\PP^1$-bundle map, $S^{[2, n]}$ is the incidence subscheme parameterizing $(\xi, \xi')\in S^{[2]}\times S^{[n]}$ such that $\xi\subset \xi'$. All the varieties appearing here have a natural real structure inherited from that of $S$. 
	Since $E$ is of odd dimension, the natural real structure is orientation reversing, hence the class $[E]$ is $\sigma$-anti-invariant; since $S^{[2, n]}$ is of even dimension, the class $[S^{[2, n]}]$ is $\sigma$-invariant. Therefore, $v$ is $\sigma$-anti-invariant. 
	
	Similarly,  consider 
	\begin{equation}
		u:=\1_{-(n-2)}\p_{-1}(\alpha)\p_{-1}(\alpha)|0\rangle,
	\end{equation}
	which is the image of $\alpha\otimes \alpha $ by the following composition of correspondences:
	\begin{equation}
		H^2(S,\ZZ)\otimes H^2(S,\ZZ) \xrightarrow{\times}  H^4(S\times S,\ZZ) \xrightarrow{[\Bl_{\Delta}(S\times S)]_*} H^4(S^{[2]}, \ZZ) \xrightarrow{[S^{[2, n]}]_*} H^4(S^{[n]}, \ZZ),
	\end{equation}
	where the first map is the exterior product, the second map is the correspondence by $\Bl_{\Delta}(S\times S)$, whose map to $S^{[2]}$ is the  quotient by the involution swapping two factors. Since all varieties appearing here have natural real structures and are of even dimensions, we see that $u$ is $\sigma$-invariant. 
	
	As a result, the following element also belonging to $H^4(S^{[n]}, \ZZ)$ (see \Cref{rmk:KeyRelation})
	\begin{equation}
		w:=\1_{-(n-2)} \m_{1,1}(\alpha)|0\rangle = \frac{1}{2}(u-v)
	\end{equation}
	satisfies that $\sigma(w)=\frac{1}{2}(u+v)=w+v$.
	
	By \Cref{thm:LiQinWang-IntegralBasis}, or more explicitly \Cref{rmk:IntegralBasisH2H4}, by extending  $\alpha$ into a basis of $H^2(S, \ZZ)$, we see that $v$ and $w$ are part of an integral basis of $H^4(S^{[n]}, \ZZ)$. The $G$-action on the submodule $\ZZ w\oplus \ZZ v$ has matrix $\begin{bmatrix}
	1 & 1 \\
	0 & -1
	\end{bmatrix}$
	hence has Comessatti characteristic 1. By \Cref{lemma:ComessattiPrimitiveSubModule}, 
	\begin{equation}
		\lambda(H^4(S^{[n]}, \ZZ))\geq \lambda(\ZZ v\oplus \ZZ w)=1.
	\end{equation}
	In particular, the natural real structure on $S^{[n]}$ is not maximal.
\end{proof}

\subsection{Natural holomorphic involutions on Hilbert schemes}
We achieve an analogous criterion for the maximality of the naturally induced holomorphic involutions on Hilbert schemes of points of regular surfaces.

\begin{theorem}
	\label{thm:NaturalHoloInv}
	Let $n\geq 2$. Let $S$ be a smooth projective surface and $\sigma$ a holomorphic involution. Assume that $H^1(S, \F2)=0$. Then  the induced involution on $S^{[n]}$ is maximal if and only if  $\sigma$ is maximal and acts on $H^2(S, \ZZ)$ trivially.
\end{theorem}
\begin{proof}
	We first prove the \textit{if} part. Assume that the G-action on $H^2(S, \ZZ)$ is trivial, we need to show the surjectivity of the natural map 
	\begin{equation}
		\label{eqn:RestrictionMap}
		H^*_G(S^{[n]}, \F2)\to    H^*(S^{[n]}, \F2).
	\end{equation} 
	By \Cref{rmk:TorsionFree}, $H^*(S^{[n]}, \ZZ)$ is 2-torsion-free and by \Cref{thm:LiQinWang-Generators}, we have the following set of generators of $H^*(S^{[n]}, \F2)$:
	\begin{enumerate}
		\item the Chern classes $c_j(\mathcal{O}_S^{[n]})\in H^{2j}(S^{[n]}, \F2)$, for $1\leq j\leq n$;
		\item the classes $\mathds{1}_{-(n-j)}\mathfrak{m}_{(1^j)}(\alpha)|0\rangle$, for $1\leq j\leq n$ and $\alpha\in H^2(S, \ZZ)$;
		\item the classes $\mathds{1}_{-(n-j)}\mathfrak{p}_{-j}(\pt)|0\rangle$, for $1\leq j\leq n$.
	\end{enumerate}
	Let us show that these generators are in the image of \eqref{eqn:RestrictionMap}:\\
	For (i), since $\mathcal{O}_S^{[n]}$ is a $G$-equivariant vector bundle, we can consider the G-equivariant Chern class $c^G_j(\mathcal{O}_S^{[n]})\in H^{2j}_G(S^{[n]}, \F2)$. Its image in $H^{2j}(S^{[n]}, \F2)$ is $c_j(\mathcal{O}_S^{[n]})$ by naturality.\\
	For (ii), since $H^1(S, \ZZ)=0$, the Leray--Serre spectral sequence gives rise to an exact sequence
	\begin{equation}
		\label{eqn:ExactSeq}
		 H^2_G(S, \ZZ)\to H^2(S, \ZZ)^G\to H^3(G, \ZZ)\to H^3_G(S, \ZZ).
	\end{equation}
	Since $\operatorname{Fix}(\sigma)\neq \emptyset$, the natural map $S_G\to \mathsf BG$ has a section. Therefore the last map in \eqref{eqn:ExactSeq} is split injective. It yields that the following map is surjective
	\begin{equation}
		H^2_G(S, \ZZ)\to H^2(S, \ZZ)^G=H^2(S, \ZZ)
	\end{equation}
	where the last equality is by our assumption. 
	Since $H^2_G(-, \ZZ)$ classifies $G$-equivariant $\mathcal{C}^{\infty}$  complex line bundles (via the equivariant first Chern class map), we conclude that for any $\alpha\in H^2(S, \ZZ)$, there exists a $G$-equivariant $\mathcal{C}^{\infty}$ complex line bundle $L_{\alpha}$ on $S$ with $c_1^G(L_\alpha)=\alpha$. By \cite[Proof of Lemma 3.5]{LiQin08},	we have :
	$$c_j(p_{1,!}(p_2^*(L_\alpha)))=\mathfrak{m}_{(1^j)}(\alpha)|0\rangle\in H^*(S^{[j]}, \F2),$$
	where $p_1$ and $p_2$ are natural maps from the universal subscheme $S^{[1,n]}\subset S^{[n]}\times S$ to $S^{[n]}$ and $S$ respectively. As $p_1$ and $p_2$ are $G$-equivariant maps, $p_{1,!}(p_2^*(L_\alpha))$ can be viewed as an element in the $G$-equivariant topological KU-theory of $S^{[n]}$. 
	Therefore, the equivariant Chern class $c^G_j(p_{1,!}(p_2^*(L_\alpha)))\in H_G^*(S^{[j]}, \F2)$ is mapped to $\mathfrak{m}_{(1^j)}(\alpha)|0\rangle$ via the left vertical arrow in the following commutativity diagram:
	\begin{equation}
		\xymatrix{
			H^*_G(S^{[j]}, \F2)\ar[r]^{[S^{[j,n]}]_*}\ar[d]& H^*_G(S^{[n]}, \F2)\ar[d]\\
			H^*(S^{[j]}, \F2)\ar[r]^{[S^{[j,n]}]_*}& H^*(S^{[n]}, \F2)
		}
	\end{equation}
	Therefore $\mathds{1}_{-(n-j)}\mathfrak{m}_{(1^j)}(\alpha)|0\rangle=[S^{[j,n]}]_*\mathfrak{m}_{(1^j)}(\alpha)|0\rangle$ is in the image of the right vertical arrow.\\
	For (iii), we have a commutative diagram
	\begin{equation}
		\xymatrix{
			H^*_G(S, \F2)\ar[r]^{[S^{[1,j]}]_*} \ar@{->>}[d]& H^*_G(S^{[j]}, \F2)\ar[r]^{[S^{[j,n]}]_*}\ar[d]& H^*_G(S^{[n]}, \F2)\ar[d]\\
			H^*(S, \F2)\ar[r]^{[S^{[1,j]}]_*} & H^*(S^{[j]}, \F2)\ar[r]^{[S^{[j,n]}]_*}& H^*(S^{[n]}, \F2)
		}
	\end{equation}
	where the left vertical arrow is surjective by maximality assumption. 
	Since $\mathds{1}_{-(n-j)}\mathfrak{p}_{-j}(\pt)|0\rangle$ is the image of $\pt\in H^4(S, \F2)$ under the composition of the bottom maps, and by the commutativity  of the diagram, it is in the image of the right vertical arrow. \\
	We proved that all the generators in (i), (ii) and (iii) are in the image of \eqref{eqn:RestrictionMap}, hence \eqref{eqn:RestrictionMap} is surjective. By \Cref{prop:MaximalityViaCohomology}, the involution on $S^{[n]}$ induced by $\sigma$ is maximal. \\

	We now show the \textit{only if} part.  The proof is similar to that of \Cref{thm:NaturalAntiHoloInv}, let us only emphasize the differences. 
    Assume the natural involution on $S^{[n]}$ induced by $\sigma$ is maximal. 
    By \Cref{thm:NonMaxHilbSurfWithoutFixPoint} (or rather \Cref{cor:FixedPointFree-NonMaximal}), $\sigma$ admits fixed points, {\it i.e.} $S^{\sigma}\neq \emptyset$. The maximality of $\sigma$ is proven by exactly the same argument as in \Cref{thm:NaturalAntiHoloInv}. In particular, 
	\begin{equation}
		H^2(S, \ZZ)= H^2(S, \ZZ)^{\sigma}\oplus H^2(S, \ZZ)^{\sigma-}.
	\end{equation}
	Assume that the action of $\sigma$ on $H^2(S, \ZZ)$ is not trivial. Choose a primitive anti-invariant element $\alpha\in H^2(S, \ZZ)^{\sigma-}$, and  extend it up to a basis of $H^2(S, \ZZ)$. Then consider the following element 
	\begin{equation}
		v:= \1_{-(n-2)}\p_{-2}(\alpha)|0\rangle.
	\end{equation}
	By definition of the Nakajima and Li--Qin--Wang operators (recalled in \Cref{subsec:Nakajima-Li-Qin-Wang-Operators}), $v$ is the image of $\alpha$ via the following composition of correspondences induced by $G$-invariant cycles
	\begin{equation}
		H^2(S,\ZZ)\xrightarrow{[E]_*} H^4(S^{[2]}, \ZZ) \xrightarrow{[S^{[2, n]}]_*} H^4(S^{[n]}, \ZZ),
	\end{equation}
	Therefore, $v$ is $\sigma$-anti-invariant. 
	
	Similarly,  consider 
	\begin{equation}
		u:=\1_{-(n-2)}\p_{-1}(\alpha)\p_{-1}(\alpha)|0\rangle,
	\end{equation}
	which is the image of $\alpha\otimes \alpha $ by the following composition of correspondences induced by $G$-invariant cycles
	\begin{equation}
		H^2(S,\ZZ)\otimes H^2(S,\ZZ) \xrightarrow{\times}  H^4(S\times S,\ZZ) \xrightarrow{[\Bl_{\Delta}(S\times S)]_*} H^4(S^{[2]}, \ZZ) \xrightarrow{[S^{[2, n]}]_*} H^4(S^{[n]}, \ZZ).
	\end{equation}
	We see that $u$ is $\sigma$-invariant. 
	
	As a result, the following element also belonging to $H^4(S^{[n]}, \ZZ)$ (see \Cref{rmk:KeyRelation})
	\begin{equation}
		w:=\1_{-(n-2)} \m_{1,1}(\alpha)|0\rangle = \frac{1}{2}(u-v)
	\end{equation}
	satisfies that $\sigma(w)=\frac{1}{2}(u+v)=w+v$.
	Then we conclude the non-maximality as in \Cref{thm:NaturalAntiHoloInv}.
\end{proof}

\subsection{Examples}
\label{subsec:Examples}
In this subsection, we apply \Cref{thm:NaturalAntiHoloInv} and \Cref{thm:NaturalHoloInv} to deduce the non-existence of maximal involutions on Hilbert schemes of certain surfaces, and also provide examples of maximal involutions on Hilbert schemes of some surfaces.

\begin{cor}[Hilbert schemes of projective plane]
	\label{cor:P2}
	Let $\sigma$ be a holomorphic or anti-holomorphic involution on $\PP^2$. If $\sigma$ is maximal, then the naturally induced involution on $(\PP^2)^{[n]}$ is maximal for any integer $n\geq 1$.
\end{cor}     
\begin{proof}
    Let $\sigma$ be an involution of $\PP^2$. 
    If $\sigma$ is holomorphic, then it sends a projective line to a projective line preserving the complex orientation, while if it is anti-holomorphic, it sends a projective line to a projective line but with the complex orientation reversed. Since the class of a projective line is a generator of $H^2(\PP^2, \ZZ)$, the result follows from  \Cref{thm:NaturalHoloInv} and  \Cref{thm:NaturalAntiHoloInv}, respectively.
\end{proof}
\begin{rmk}
\Cref{cor:P2} can be applied to produce examples of maximal holomorphic involutions and anti-holomorphic involutions on $(\PP^2)^{[n]}$.
	\begin{enumerate}
	    \item[(i)]  The holomorphic involution $[T_0:T_1:T_2]\mapsto [-T_0:T_1:T_2]$ on $\PP^2$ is maximal. Indeed, the fixed locus is $\{[1:0:0]\}\cup (T_0=0)$, the union of a point and a projective line. By \Cref{cor:P2}, we get a maximal holomorphic involution on $(\PP^2)^{[n]}$.
		\item[(ii)] The natural real structure $[T_0:T_1:T_2]\mapsto [\overline{T_0}:\overline{T_1}:\overline{T_2}]$ on $\PP^2$ is maximal. Indeed, the real locus is $\RR\PP^2$. By \Cref{cor:P2}, the natural real structure on $(\PP^2)^{[n]}$ is maximal. 
	\end{enumerate}
    In fact, it is easy to see, and well known, that those are the \textit{only} examples of (anti-)holomorphic involutions on $\PP^2$ up to equivalence: any holomorphic involution of $\PP^2$ is projectively equivalent to the reflection shown in (i), and any real structure on $\PP^2$ is conjugate to the natural one shown in (ii).
\end{rmk}

Let us turn to some constructions of non-maximal involutions. 
We first point out the following application to hyper-K\"ahler geometry. It will be generalized in \Cref{sec:non_existence_maximal_K3[n]} (see \Cref{thm:main:NonMaxHKnOdd}, \Cref{cor:main:NoMaxRealStructureHKnOdd} and \Cref{thm:main:NonMaxHKnSymp}).

\begin{cor}[Hilbert schemes of K3]
	\label{cor:NaturalHilbertK3}
	Let  $\sigma$ be a (non-trivial) holomorphic or anti-holomorphic involution on a K3 surface $S$. Then for any $n\geq 2$, the induced involution on $S^{[n]}$ is not maximal.
\end{cor}
\begin{proof}
Assume for contradiction that the induced involution on $S^{[n]}$ is maximal.

We first reduce the problem to the projective case.
If $\sigma$ is holomorphic anti-symplectic, then \(S\) is projective (see \cite[Theorem 0.1]{nikulin1979finite}).
If $\sigma$ is holomorphic symplectic or anti-holomorphic, then $S$ might not be projective but it is known that it can be deformed together with the involution to a projective K3 surface. 
For holomorphic symplectic involutions, it follows from the connectedness of their moduli space (obtained by Nikulin in {\it loc.~cit.} and existence of such involutions on, say, the Fermat quartic in $\PP^3$).
For anti-holomorphic involutions, their moduli space is disconnected, but a similar approach can be applied. Namely, by \cite[Theorem 13.8.1]{RealEnriques}, the deformation type of a real K3 surface depends only on the isomorphism class of the lattice $H^2(S, \ZZ)$ with an involution. Thus, it is sufficient to notice that the list of these isomorphism classes (summarized, for example, in \cite[Theorem 8.4.2]{RealEnriques}) coincides with Nikulin's list
\cite{nikulin1979integral} for those classes that are realized by real quartic K3 surfaces in $\PP^3$ (or to notice that, as it follows from \cite{rokhlin1978complex}, the number of classes realized by real K3 surfaces obtained as a double plane branched in a real curve of degree 6 is not smaller).
Thus, as maximality is a deformation-invariant notion, we can assume $S$ to be projective in all the cases, and can apply \Cref{thm:NaturalHoloInv} and \Cref{thm:NaturalAntiHoloInv}.

If $\sigma$ is holomorphic, \Cref{thm:NaturalHoloInv} implies that 
    $\sigma$ must act trivially on $H^2(S, \ZZ)$. But then, by \cite{Sapiro-Shafarevich-K3Torelli}, $\sigma$ is the identity.

    If \(\sigma\) is anti-holomorphic, then by \Cref{thm:NaturalAntiHoloInv}, the assumption that $S^{[n]}$ is maximal implies that $S$ is maximal with connected real locus, or equivalently, by \cite[Proposition 4.4]{FuMaximalReal}, the whole $H^2(S, \ZZ)$ is $\sigma$-anti-invariant. Now we have two ways to conclude:
    either use the fact that any maximal K3 surface has disconnected real locus by \cite[Theorem 8.4.1]{RealEnriques} (see \cite[Lemma 5.5]{Kharlamov-Rasdeaconu-HilbertSquare} for a generalization), or argue that by \Cref{prop:HKRotation}, up to changing the complex structure on $S$ by a hyper-K\"ahler rotation, $\sigma$ becomes a holomorphic anti-symplectic involution, hence it must preserve the K\"ahler cone and cannot be $-\id$ on $H^2(S, \RR)$.
\end{proof}

In fact, the statement for anti-holomorphic involutions in \Cref{cor:NaturalHilbertK3} holds more generally for surfaces with non-vanishing geometric genus. See the following generalization of \cite[Corollary 1.3]{Kharlamov-Rasdeaconu-HilbertSquare}. 

\begin{cor}
\label{cor:NaturalHilbPg>0}
    Let $S$ be a smooth projective $\RR$-surface with $H^{2,0}(S)\neq 0$ and $H^1(S, \F2)=0$. Then for any $n\geq 2$, the $n$-th Hilbert scheme $S^{[n]}$, equipped with the naturally induced real structure, is not maximal.
\end{cor}
\begin{proof}
    Since $H^{2,0}(S)$ is mapped to $H^{0,2}(S)$ by any anti-holomorphic involution, the induced action on $H^2(S, \CC)$ cannot be $-\id$.
    Applying \Cref{thm:NaturalAntiHoloInv}, we see that the induced anti-holomorphic involution on the $n$-th Hilbert scheme is not maximal for any $n\geq 2$.
\end{proof}

To get more examples, we remark that the conditions in \Cref{thm:NaturalAntiHoloInv} and \Cref{thm:NaturalHoloInv}  behave well under birational transformations:

\begin{prop}[Blow-ups]
	\label{prop:BlowUp}
	Let $S$ be a smooth projective surface equipped with a holomorphic (resp.~anti-holomorphic) involution $\sigma$. Let $P\in S$ be a fixed point of $\sigma$. Let $\tilde{\sigma}$ be the holomoprhic (resp.~anti-holomorphic) involution on $\Bl_PS$ lifting $\sigma$. Then 
	\begin{enumerate}
		\item $\sigma$ is maximal if and only if $\tilde{\sigma}$ is maximal;
		\item $\sigma$ acts as $\id$ (resp.~$-\id$) on $H^2(S, \ZZ)$ if and only if the same holds for $\tilde{\sigma}$.
	\end{enumerate}
\end{prop}
\begin{proof}
	(i). If $\sigma$ is holomorphic and $P$ belongs to a fixed curve $C$, then the fixed locus of $\tilde{\sigma}$ around the exceptional divisor $E$ is the union of the strict transform of $C$ (which is isomorphic to $C$) and another point on $E$. Therefore, the total $\F2$-Betti numbers of both the surface and the fixed locus increase by 1. Hence the maximality of $\sigma$ is equivalent to that of $\tilde{\sigma}$.
	
	If $\sigma$ is holomorphic and $P$ is an isolated fixed point, then the fixed locus of $\tilde{\sigma}$ around $E$ is $E$ itself.  Therefore, the total $\F2$-Betti numbers of both the surface and the fixed locus increase by 1, and we have again the equivalence between the maximalities of $\sigma$ and $\tilde{\sigma}$.
	
	If $\sigma$ is anti-holomorphic, then the fixed locus of $\sigma$ is a 2-dimensional manifold $M$, while the fixed locus of $\tilde{\sigma}$ is $M\# \RR\PP^2$, where $\#$ denotes the connected sum. Therefore, the total $\F2$-Betti number of the fixed locus increases by 1, and the maximalities of $\sigma$ and $\tilde{\sigma}$ are equivalent. 
	
	For (ii), it is enough to notice that the lifted involution preserves (resp.~reverses) the orientation of the exceptional divisor in the holomorphic (resp.~anti-holomorphic) case.
\end{proof}

\begin{rmk}
Thanks to \Cref{prop:BlowUp}, starting from a smooth projective surface $S$ with $H^1(S, \F2)=0$ together with a maximal holomorphic or anti-holomorphic involution $\sigma$, after a successive blow-ups along fixed points of (lifted) involutions, we get a new surface $S'$ equipped with involution $\sigma'$, then the induced involution on the $n$-th Hilbert schemes of points on $S'$ is maximal for $n\geq 2$ if and only if the same holds for the Hilbert schemes of $S$. In this way, from
\Cref{cor:P2}, \Cref{cor:NaturalHilbertK3} and \Cref{cor:NaturalHilbPg>0}, we can produce many maximal and non-maximal involutions on Hilbert schemes of surfaces. This generalizes \cite[Corollary 1.4]{Kharlamov-Rasdeaconu-HilbertSquare} to arbitrarily higher dimensions.
\end{rmk}

\section{Nonexistence of maximal brane involutions on hyper-K\"ahler manifolds of $\K3^{[n]}$-type}\label{sec:non_existence_maximal_K3[n]}
In this section, we focus on the existence problem of maximal (holomorphic as well as anti-holomorphic) involutions on compact hyper-K\"ahler manifolds. In contrast to \Cref{cor:NaturalHilbertK3}, we deal more generally with hyper-K\"ahler manifolds that are only deformation equivalent to Hilbert schemes of K3 surfaces, and involutions beyond the naturally induced ones. The main results have been stated 
in Introduction as \Cref{thm:main:NonMaxHKnOdd}, \Cref{cor:main:NoMaxRealStructureHKnOdd} and \Cref{thm:main:NonMaxHKnSymp}.

\subsection{Holomorphic symplectic involutions}
Here we treat (BBB)-brane involutions and deduce \Cref{thm:main:NonMaxHKnSymp} from results obtained in \Cref{sec:Natural}. 

\begin{proof}[Proof of \Cref{thm:main:NonMaxHKnSymp}]
Thanks to \cite[Theorem 1.0.3 and Theorem 2.0.6]{Kam_Mon_Obl-SymplecticAutomorphisms}\footnote{One can also directly use \cite[Corollary 3.3]{Kam_Mon_Obl}, whose proof relies on \cite[Theorem 3.2]{Kam_Mon_Obl}. Although the statement of \cite[Theorem 3.2]{Kam_Mon_Obl} for Kummer type hyper-K\"ahler varieties does not hold, it is correct in the case of K3$^{[n]}$-type.}, for any symplectic involution $\tau$ on a hyper-K\"ahler manifold $X$ of K3\(^{[n]}\) type, the pair \((X,\tau)\) is deformation equivalent to a pair \((S^{[n]},\sigma^{[n]})\) for a K3 surface \(S\) and natural involution $\sigma^{[n]}$ induced by a symplectic involution $\sigma$ on \(S\).
As a consequence, the fixed locus of $\tau$ is diffeomorphic to the fixed locus of $\sigma^{[n]}$, which is never maximal by \Cref{cor:NaturalHilbertK3}, 
unless the involution is trivial. 
\end{proof}
\begin{rmk}
In \Cref{subsec:BBB-not-optimal}, we give an alternative proof of this result which also covers the case of hyper-Kähler manifolds of OG6 type; see \Cref{cor:NoMaxBBB}.
\end{rmk}
\subsection{Anti-symplectic involutions}
The goal of this section is to prove \Cref{thm:main:NonMaxHKnOdd}.

Let $X$ be a compact hyper-K\"ahler manifold of $\K3^{[n]}$-type with $n \geq 2$. \textit{Assume for contradiction that $\sigma$ is a maximal holomorphic anti-symplectic involution of $X$}.
Denote by $H^*(X, \ZZ)^\sigma$ the subgroup of $\sigma$-invariant elements and by $H^*(X, \ZZ)^{\sigma-}$ the subgroup of $\sigma$-anti-invariant elements.
By \Cref{prop:MaximalityViaCohomology} and \Cref{lemma:EquivalentConditionsForSplitting}, we have a direct sum decomposition of cohomology group for each degree $k$:
\begin{equation}
	H^k(X,\ZZ)= H^k(X, \ZZ)^{\sigma}\oplus H^k(X, \ZZ)^{\sigma-}.
\end{equation}
In particular, equipping the second cohomology with the 
Beauville--Bogomolov--Fujiki quadratic form \cite{Beauville}, 
we have an \textit{orthogonal} direct sum decomposition of 
integral lattices:
\begin{equation}
	\label{eqn:SplittingH2}
	H^2(X,\ZZ)= H^2(X, \ZZ)^{\sigma}\oplus H^2(X, \ZZ)^{\sigma-}.
\end{equation}
\textit{In the rest of this section, we always assume \eqref{eqn:SplittingH2}.}

Denote by $L:=H^2(X, \ZZ)$, which is isometric to $U^{\oplus 3}\oplus E_8(-1)^{\oplus 2}\oplus \langle-2n+2\rangle$. 
We have the following classification of its eigen-sub-lattices:
\begin{lemma}
	\label{lemma:LatticeClassification} 
	Under the assumption \eqref{eqn:SplittingH2}, there are only four cases:
    
	\begin{enumerate}
		\item[Case 1.]  $L^{\sigma}$ is unimodular and $A_{L^{\sigma-}}\cong \ZZ/(2n-2)\ZZ$. In this case, we have the following classification:
		\begin{equation*}
			\begin{cases}
				L^{\sigma}\cong U \\
				L^{\sigma-}\cong U^{\oplus 2}\oplus E_8(-1)^{\oplus 2} \oplus \langle-2n+2\rangle
			\end{cases}
			\text{or }
			\begin{cases}
				L^{\sigma}\cong U \oplus E_8(-1)\\
				L^{\sigma-}\cong U^{\oplus 2}\oplus E_8(-1) \oplus \langle-2n+2\rangle
			\end{cases}
			\text{or }
			\begin{cases}
				L^{\sigma}\cong U \oplus E_8(-1)^{\oplus 2}\\
				L^{\sigma-}\cong U^{\oplus 2} \oplus \langle-2n+2\rangle.
			\end{cases}
		\end{equation*}
		
        \item[Case 2.] $L^{\sigma-}$ is unimodular and $A_{L^{\sigma}}\cong \ZZ/(2n-2)\ZZ$. In this case, we have the following classification: 
		\begin{equation*}
			\begin{cases}
				L^{\sigma}\cong U\oplus \langle-2n+2\rangle\\
				L^{\sigma-}\cong U ^{\oplus 2}\oplus E_8(-1)^{\oplus 2} 
			\end{cases}
			\text{ or }
			\quad
			\begin{cases}
				L^{\sigma}\cong U \oplus E_8(-1)\oplus \langle-2n+2\rangle\\
				L^{\sigma-}\cong U^{\oplus 2}\oplus E_8(-1)
			\end{cases}
			\text{ or }
			\quad
			\begin{cases}
				L^{\sigma}\cong U \oplus E_8(-1)^{\oplus 2}\oplus \langle-2n+2\rangle\\
				L^{\sigma-}\cong U^{\oplus 2}. 
			\end{cases}
		\end{equation*}
        \item[Case 3.] $A_{L^{\sigma}}\cong\ZZ/2\ZZ$ and $A_{L^{\sigma-}}\cong\ZZ/(n-1)\ZZ$. This case can only happen when $n\geq 4$ is an even integer.
        \item[Case 4.] $A_{L^{\sigma}}\cong\ZZ/(n-1)\ZZ$ and $A_{L^{\sigma-}}\cong\ZZ/2\ZZ$. This case can only happen when $n\geq 4$ is an even integer.
	\end{enumerate}
\end{lemma}
\begin{proof}
	Most of the statements can be deduced from \cite[Proposition 2.8]{CamereCattaneoCattaneo}. As our case is simplified thanks to the assumption \eqref{eqn:SplittingH2} (which is amount to the condition $a=0$ in {\it loc.~cit.}), we give a direct proof for convenience of the reader. 
	
	Since $\sigma$ is holomorphic, it must preserve some K\"ahler class, hence the signature of $L^{\sigma}$ is $(\geq 1, -)$. Since $\sigma$ is anti-symplectic, by the Hodge--Riemann bilinear relations applied to holomorphic 2-forms, the signature of $L^{\sigma-}$ is $(\geq 2, -)$. As a result, 
	\begin{equation}
		\sgn(L^{\sigma})=(1, -) \quad\text{and }\quad  \sgn(L^{\sigma-})=(2, -).
	\end{equation}
	We denote by $\overline{\sigma}$ the action of $\sigma$ on the discriminant group $A_L\cong \ZZ/(2n-2)\ZZ$.
	By Markman \cite[Lemma 9.2]{Markman-SurveyTorelli}, $$\overline{\sigma}=\pm \id.$$
	Hence $\sigma$ or $-\sigma\in \widetilde{O}(L):=\{\phi\in O(L)~|~ \bar\phi=\id \in O(A_L)\}$. 
    
    We embed $(L,\sigma)$ if $\sigma\in \tilde O(L)$, resp.  $(L,-\sigma)$ if $-\sigma\in \tilde O(L)$, into $(\tilde{L}:=U^{\oplus 4}\oplus E_8(-1)^{\oplus 2}$, $\tilde\sigma)$ with $\tilde\sigma=\id$ on $L^\perp\subset \tilde L$ in such a way that $\tilde L^{\tilde\sigma -}=L^{\sigma -}$ if $\sigma\in \tilde O(L)$ and $\tilde L^{\tilde\sigma -}=L^{\sigma}$ otherwise. This is achieved by the gluing of the lattices $L$ and $\langle 2n-2\rangle$ up to an even unimodular lattice ({\it cf.} \cite[Theorem 1.6.1, Corollary 1.5.2]{nikulin1979integral}). The result of the gluing is isomorphic to $U^{\oplus 4}\oplus E_8(-1)^{\oplus 2}$ as it follows from the classification of indefinite unimodular lattices by rank, signature and parity (see, {\it e.g.}, \cite[Chapter II]{MilnorHusemoller-LatticeBook}). Since $\tilde L$ is unimodular and even, we conclude that $A_{\tilde L^{\tilde \sigma -}}$,
    and hence $A_{L^{\sigma-}}$ if $\sigma\in \tilde O(L)$ or $A_{L^{\sigma}}$  otherwise, is isomorphic to $(\ZZ/2\ZZ)^{\oplus m}$ for some $m\in\ZZ_{\ge 0}$. In addition, by assumption (\ref{eqn:SplittingH2}) we have \(L=L^{\sigma}\oplus L^{\sigma -}\), hence \(A_{L^{\sigma}}\oplus A_{L^{\sigma -}}\cong A_L\), where \(A_L\cong \mathbb{Z}/(2n-2)\mathbb{Z}\). All these implies the following:

    \begin{itemize}
        \item If \(n=2\) or \(n\geq 3\) is odd,  taking into account that $\mathbb{Z}/(2n-2)\mathbb{Z}$ does not admit a non-trivial direct sum decomposition with a summand of the form $(\ZZ/2\ZZ)^{\oplus m}$,  we obtain that $A_{L^{\sigma}}$ or $A_{L^{\sigma -}}$ is trivial, i.e.~$L^{\sigma}$ or $L^{\sigma-}$ is unimodular. So we are in Case 1 or Case 2 for such $n$.
        \item If \(n\geq 4\) is even, then the only non-trivial direct sum decomposition of $\mathbb{Z}/(2n-2)\mathbb{Z}$ with one summand being 2-elementary is the decomposition $\mathbb{Z}/(2n-2)\mathbb{Z}\cong \mathbb{Z}/(n-1)\mathbb{Z}\oplus \mathbb{Z}/2\mathbb{Z}$. Hence, for such $n$, if we are not in Case 1 or Case 2 , then we are in Case 3 or Case 4.
    \end{itemize}
    It remains to classify the lattices in Case 1 and Case 2:	
	\begin{enumerate}
		\item[Case 1.] If $L^{\sigma}$ is unimodular, since it is even and has index $(1,-)$, by classification of such lattices (see \cite[Chapter II]{MilnorHusemoller-LatticeBook}), it is determined by its rank and signature, and the signature is divisible by 8. Hence $$L^{\sigma}\cong U\oplus E_8(-1)^{\oplus i}$$ with $i=0, 1$ or 2.
		Consequently, $L^{\sigma-}$ is even, of index $(2, -)$, with length $\ell(A_{L^{\sigma-}})=\ell(\ZZ/(2n-2)\ZZ)=1$ and rank $5, 13$ or 21. By Nikulin \cite[Corollary 1.13.3]{nikulin1979integral}, $L^{\sigma-}$ is determined by its rank, index and discriminant form. Then the classification follows. 
		\item[Case 2.] Similarly, if $L^{\sigma-}$ is unimodular, again 
        since it is even,  its signature is divisible by 8. The only possibilities are $\sgn(L^{\sigma-})=(2,2), (2,10) \text{ or } (2, 18)$; in particular it is indefinite. Therefore 
		$$L^{\sigma-}\cong U^{\oplus 2}\oplus E_8(-1)^{\oplus i}$$
		with $i=0, 1$ or 2. Hence $L^{\sigma}$ is even, indefinite, with length 1, and with rank at least 3. Applying Nikulin \cite[Corollary 1.13.3]{nikulin1979integral}, we conclude the classification.
	\end{enumerate}
\end{proof}

Before proceeding further with the proof, we recall the following generalization of Eichler's criterion (see \cite[\S 10]{Eichler} for the classical version). For a vector \(v\) in a lattice \(\Lambda\), we denote by \(v^*\) the class \(\frac{v}{\div(v)}\) in the discriminant group $A_\Lambda$, where $\div(v)$ is the divisibility of $v$.  
\begin{theorem}[{\cite[Proposition 3.3]{gritsenko2009abelianisation}}]
\label{thm:Eichler}
    Let $\Lambda$ be an even lattice containing at least two orthogonal copies of hyperbolic planes. Let $\widetilde{O}^{+}(\Lambda)$ be the group of isometries of $\Lambda$ of trivial spinor norm that act trivially on the discriminant group $A_\Lambda$.
    Then the $\widetilde{O}^{+}(\Lambda)$-orbit of a primitive vector $v\in \Lambda$ is determined by the integer $v^2$ and the class $v^*$ in the discriminant group $A_\Lambda$.
\end{theorem}

\begin{lemma}\label{cor:deltaEigen}
     Assume $L^{\sigma}$ or $L^{\sigma-}$ to be unimodular (which is always the case when $n\geq 3$ is an odd integer or $n=2$), and assume that the class \(v^*\) generates the discriminant group \(A_{L}\) for a primitive vector \(v\in L\) with \(v^2=2-2n\) and \(\div(v)=2n-2\). Then there exists a primitive element \(\epsilon\in L^{\sigma-}\), or \(\epsilon\in L^{\sigma}\) respectively, with \(\epsilon^2= v^2\) and \(\div(\epsilon)=\div(v)\) such that \(\epsilon^*=v^*\) in \(A_L\).
\end{lemma}
   
\begin{proof}
    Let us denote by \(L'\) the sublattice \(L^{\sigma}\) or \(L^{\sigma-}\), depending on \(L^{\sigma-}\) or \(L^{\sigma}\) is unimodular. From the classification in \Cref{lemma:LatticeClassification}, it is clear that there exists a primitive element \(x\in L'\) with \(x^2=2-2n\) and \(\div(x)=2n-2\) such that 
    \(A_L\) is generated by $x^*:=\frac{x}{\div(x)}$. On the other hand, by hypothesis \(A_L\) is also generated by $v^*$, hence there exists an integer $k$ that is coprime to $2n-2$, such that 
     \(v^*=kx^*\) in \(A_L\). Therefore, we have the following equality in $\mathbb{Q}/2\mathbb{Z}$, since the lattice \(L\) is even:
     \[(v^*)^2=\frac{-1}{2n-2}=\frac{-k^2}{2n-2}=(k x^*)^2.\]
     This readily implies that \(\frac{k^2-1}{2n-2}\in 2\mathbb{Z}\). 
     Again by the classification in \Cref{lemma:LatticeClassification}, there is always a copy of the hyperbolic plane $U$ in \(x^{\perp_{L'}}\) as a direct summand. As the hyperbolic plane quadratic form represents all even integers, there exists \(w\in L'\) such that $w\perp x$ and \(w^2=\frac{k^2-1}{2n-2}\). 
     
     We claim that the following element in $L'$ satisfies the desired properties:
     \[\epsilon:=kx+(2n-2)w.\]
     Indeed, since  
     $L'=\ZZ\cdot x\oplus x^{\perp_{L'}}$
     and $k$ is coprime to $2n-2$ and $x$ is primitive, $\epsilon$ is a primitive vector with     
     \(\div(\epsilon)=2n-2\). It is also straightforward to compute that \(\epsilon^2=2-2n\). Finally, we observe that \(\epsilon^*=kx^*+w=kx^*=v^*\) in $A_L$.
\end{proof}

\begin{lemma}\label{lemma:MarkingParallelTransport}
    Notation is as above, in particular \(L=H^2(X,\mathbb{Z})\). If $L^{\sigma}$ or $L^{\sigma-}$ is unimodular,
 then there exists a graded ring isomorphism 
	\begin{equation}
		\phi\colon H^*(X, \ZZ)  \xrightarrow{\cong} H^*(S^{[n]}, \ZZ),
	\end{equation}
	such that $S$ is a (any) K3 surface, and there exists a primitive element \(\epsilon\in H^2(X,\mathbb{Z})^{\sigma-}\), or \(\epsilon\in H^2(X,\mathbb{Z})^{\sigma}\) respectively, with \(\epsilon^2=2-2n\) and divisibility \(2n-2\) which is sent to $\delta:=\frac{1}{2}[E]$, half of the  class of the exceptional divisor in $S^{[n]}$. Moreover, $\phi$ can be chosen to be a parallel transport operator.
\end{lemma}
\begin{proof}
    For any K3 surface \(S\), we have that \(X\) is deformation equivalent to \(S^{[n]}\) and hence there is a parallel transport operator inducing an isometry \(f\colon H^2(X,\mathbb{Z})\cong H^2(S^{[n]},\mathbb{Z})\). Consider the class \(\delta\in H^2(S^{[n]},\mathbb{Z})\) corresponding to \(\frac{1}{2}[E]\) with \(E\) the exceptional divisor of \(S^{[n]}\).
    The element \(f^{-1}(\delta)^*\) generates the discriminant group of \(L=H^2(X,\mathbb{Z})\), hence by \Cref{cor:deltaEigen} there exists an element \(\epsilon\in L^{\sigma-}\), or \(\epsilon\in L^{\sigma}\) respectively, with the same square and divisibility as $f^{-1}(\delta)$ and such that \(\epsilon^*=f^{-1}(\delta)^*\) in $A_{H^2(X, \ZZ)}$.

Therefore, by \Cref{thm:Eichler}, we can find an orientation-preserving isometry acting trivially on the discriminant group \(g\in\widetilde{O}^+(L)\) such that \(g(\epsilon)=f^{-1}(\delta)\). Markman proved that \(\Mon^2(X)\cong W(H^2(X,\mathbb{Z})):=\{h\in O^+(H^2(X, \ZZ))~|~ \bar h=\pm\id \in O(A_{H^2(X, \ZZ)})\}\) if \(X\) is of K3\(^{[n]}\)-type, see for example \cite{Markman-SurveyTorelli}. This means that the isometry \(g\) is induced by parallel transport and hence it is also the case of the composition \(\phi:=f\circ g\colon H^2(X,\mathbb{Z})\to H^2(S^{[n]},\mathbb{Z})\), which is then the restriction of a ring isomorphism that we call again \(\phi\colon H^*(X, \ZZ)  \xrightarrow{\cong} H^*(S^{[n]}, \ZZ)\).
\end{proof}

\begin{proof}[Proof of \Cref{thm:main:NonMaxHKnOdd}]
        Keep the notation as before. We separate the proof according to the four cases in \Cref{lemma:LatticeClassification}, firstly in Case 1 or 2, then in Case 3 or 4. 
        
	\textbf{In Case 1 or Case 2} of \Cref{lemma:LatticeClassification}: then $L^{\sigma}$ or $L^{\sigma-}$ is unimodular. By \Cref{cor:deltaEigen} and \Cref{lemma:MarkingParallelTransport}, we can find a graded ring isomorphism \(\phi\colon H^*(X,\mathbb{Z})\cong H^*(S^{[n]},\mathbb{Z})\) sending $\epsilon$ to $\delta$. We set 
	\[\iota:= \phi\circ \sigma \circ \phi^{-1} \in \Aut(H^*(S^{[n]},\mathbb{Z}))\]
	We will still denote by \(\sigma\) and $\iota$ their restrictions to \(H^2(- ,\mathbb{Z})\) when it does not lead to confusion. 
	Notice that
    $\iota\colon H^2(S^{[n]}, \ZZ) \to H^2(S^{[n]}, \ZZ)$ is an orientation-preserving isometry acting by $\pm \id$ on the discriminant group, i.e. \(\iota \in W(H^2(S^{[n]},\mathbb{Z}))\cong \Mon^2(S^{[n]})\cong \Mon(S^{[n]})\). 
	
    The action of \(\sigma\) on \(H^*(X,\mathbb{F}_2)\) is trivial if and only if the action of \(\iota\) on \(H^*(S^{[n]},\mathbb{F}_2)\) is. Hence we are reduced to the study of the monodromy involution \(\iota\) on the cohomology of \(S^{[n]}\). The goal is to show that Comessatti characteristic of $\iota$ on $H^4(S^{[n]}, \ZZ)$ is at least 1. 
	
	From the classification in Case 1 and Case 2 in \Cref{lemma:LatticeClassification}, we observe that \(\epsilon^{\perp}\cap H^2(X, \ZZ)^{\sigma}\not= 0\) or respectively \(\epsilon^{\perp}\cap H^2(X, \ZZ)^{\sigma-}\not= 0\), so that we are able to pick a non-zero primitive element \(\alpha\in \delta^{\perp}\cap H^2(S^{[n]},\mathbb{Z})^{\iota-}\) when $\iota(\delta)=\delta$ (Case 2), or $\alpha\in \delta^{\perp}\cap H^2(S^{[n]},\mathbb{Z})^{\iota}$ when $\iota(\delta)=-\delta$ (Case 1). Extend $\alpha$ to a basis of \(H^2(S,\mathbb{Z})\).
	Consider the basis element 
	\[v:= \1_{-(n-2)}\p_{-2}(\alpha)|0\rangle,\]
	from \Cref{thm:LiQinWang-IntegralBasis} and the element 
	\[u:=\1_{-(n-2)}\p_{-1}(\alpha)\p_{-1}(\alpha)|0\rangle.\]
	Since $\iota$ is monodromy operator, one can apply \Cref{cor:property_3} and \Cref{lemma:Property_2} to perform the following computation:
    \begin{itemize}
        \item if $\iota(\delta)=\delta$, then denoting again by $\iota$ its restriction to $H^2(S,\ZZ)$, we have 
    \begin{align*}
        \iota(u)&=\rho_n(\iota, 1)(\1_{-(n-2)}\p_{-1}(\alpha)\p_{-1}(\alpha)|0\rangle)\\
                &=\1_{-(n-2)}\p_{-1}(\iota(\alpha))\p_{-1}(\iota(\alpha))|0\rangle\\
                &=\1_{-(n-2)}\p_{-1}(-\alpha)\p_{-1}(-\alpha)|0\rangle\\
                &=u,
    \end{align*}
    and
    \begin{align*}
        \iota(v)&=\rho_n(\iota, 1)(\1_{-(n-2)}\p_{-2}(\alpha)|0\rangle)\\
                &=\1_{-(n-2)}\p_{-2}(\iota(\alpha))|0\rangle\\
                &=\1_{-(n-2)}\p_{-2}(-\alpha)|0\rangle\\
                &=-v.
    \end{align*}
    \item if $\iota(\delta)=-\delta$, then $\tau(\iota)=-1$. Denoting again by $\iota$ its restriction to $H^2(S,\ZZ)$, we have (note that the degree operator $D$ is $\id$ since we are using its action on degree-4 cohomology):
    \begin{align*}
        \iota(u)&=\rho_n(\iota, -1)(\1_{-(n-2)}\p_{-1}(\alpha)\p_{-1}(\alpha)|0\rangle)\\
                &=\rho_n(-\iota, 1)(\1_{-(n-2)}\p_{-1}(\alpha)\p_{-1}(\alpha)|0\rangle)\\
                &=\1_{-(n-2)}\p_{-1}(-\iota(\alpha))\p_{-1}(-\iota(\alpha))|0\rangle\\
                &=\1_{-(n-2)}\p_{-1}(-\alpha)\p_{-1}(-\alpha)|0\rangle\\
                &=u,
    \end{align*}
    and
    \begin{align*}
        \iota(v)&=\rho_n(\iota, -1)(\1_{-(n-2)}\p_{-2}(\alpha)|0\rangle)\\
                &=\rho_n(-\iota, 1)(\1_{-(n-2)}\p_{-2}(\alpha)|0\rangle)\\
                &=\1_{-(n-2)}\p_{-2}(-\iota(\alpha))|0\rangle\\
                &=\1_{-(n-2)}\p_{-2}(-\alpha)|0\rangle\\
                &=-v.
    \end{align*}
    \end{itemize}
   We see that in any case, we have
    \(\iota(u)=u\) and  \(\iota(v)=-v\).
	
	As a result, the following element (see \Cref{rmk:KeyRelation})
	\begin{equation}
		w:=\1_{-(n-2)} \m_{1,1}(\alpha)|0\rangle = \frac{1}{2}(u-v).
	\end{equation}
	satisfies that $\iota(w)=\frac{1}{2}(u+v)=w+v$.
	
	By \Cref{thm:LiQinWang-IntegralBasis}, or more explicitly \Cref{rmk:IntegralBasisH2H4}, $v$ and $w$ are part of an integral basis of $H^4(S^{[n]}, \ZZ)$. The $G$-action on $\ZZ w\oplus \ZZ v$ has matrix $\begin{bmatrix}
	1 & 1 \\
	0 & -1
	\end{bmatrix}$,
	hence has Comessatti characteristic 1. By \Cref{lemma:ComessattiPrimitiveSubModule}, 
	\begin{equation}
		\lambda(H^4(S^{[n]}, \ZZ), \iota)\geq \lambda(\ZZ v\oplus \ZZ w, \iota)=1.
	\end{equation}
	The ring isomorphism \(\phi\colon H^*(X,\mathbb{Z})\cong H^*(S^{[n]},\mathbb{Z})\) is equivariant by construction. Therefore \(\lambda(H^4(X, \ZZ), \sigma)\geq1\). In particular, the involution \(\sigma\) is not maximal, which is a contradiction. The theorem is proved in Case 1 or 2.\\
    
    \textbf{In Case 3} of \Cref{lemma:LatticeClassification} (this can only happen when $n\geq 4$ is an even integer): then $A_{L^{\sigma}}\cong \ZZ/2\ZZ$ and $A_{L^{\sigma-}}\cong \ZZ/{(n-1)\ZZ}$, hence $\tau(\sigma)=-1$.
    We will use notations as in \Cref{sec:HK-Monodromy}. In particular, $$Q(X, \ZZ):=H^4(X, \ZZ)/\Sym^2H^2(X, \ZZ),$$ whose most important properties are summarized in \Cref{thm:Markman-Q}. We set $L=H^2(X, \ZZ)$ and
    \begin{equation}
        \gamma:=\frac{1}{2}\overline{c}_2(X)\in Q(X,\ZZ) \quad\quad \text{ and  } \quad\quad L':=\gamma^{\perp}=\overline{c}_2(X)^{\perp}\subset Q(X,\ZZ).
    \end{equation}
    $L$ and $L'$ are equipped with induced involutions, both denoted by $\sigma$ (to be precise, $\sigma_2$ on $L$ and the restriction of $\sigma_Q$ on $L'$).
    By \Cref{thm:Markman-Q}, there is a Hodge isometry $e\colon L\xrightarrow{\cong} L'$ such that for any $\alpha\in L$,
    \begin{equation}
    \label{eqn:eAndsigma-1}
    \sigma(e(\alpha))=-e(\sigma(\alpha)).    
    \end{equation}
     Let us choose one $U$-summand in \(Q(X,\ZZ)\cong E_8(-1)^{\oplus 2}\oplus U^{\oplus 4}\) and pick in it two orthogonal primitive elements $u,v$ such that $u^2=2n-2$,  $v^2=2-2n$ and $u+v$ is divisible by $2n-2$. According to \cite[Theorem 6]{wall1962orthogonal}, there exists an isometry
    $\phi : Q(X,\ZZ) \to Q(X,\ZZ) $ that sends $u$ to $\gamma$.
    We put $\delta'=\phi(v)$ and define $w:=\frac{\delta'+\gamma}{2n-2}\in Q(X,\ZZ)$. In particular,  \(\delta'^2=2-2n\), \(\delta'\in L'\) and the divisibility of \(\delta'\) in \(L'\) is \(2n-2\).

Let $\delta:=e^{-1}(\delta')\in L$, which has the same square and 
divisibility as $\delta'$. By the hypothesis (assumed for contradiction) $L=L^{\sigma}\oplus L^{\sigma-}$, we  can write 
    \begin{equation}
        \delta=\delta_++\delta_-
     \end{equation}
with $\delta_+\in L^\sigma$ and $\delta_-\in L^{\sigma-}$.
Combining the hypothesis $A_{L^{\sigma}}\cong \ZZ/2\ZZ$ with the fact that the pairing of $\frac{\delta_+}{2n-2}$ with any element in $L^{\sigma}$ is integral (as $\operatorname{div}(\delta)=2n-2$ and $\delta_{-}\perp L^{\sigma}$),  
we see that $x:=\frac{\delta_+}{n-1}\in L^{\sigma}$ since $A_{L^{\sigma}}\cong \ZZ/2\ZZ$. Similarly $y:=\frac{\delta_-}{2}\in L^{\sigma-}$. That is,
\begin{equation}
    \delta=(n-1)x+2y.
\end{equation}
As $\delta$ is primitive, $x$ is not divisible by 2 and $y$ is not divisible by $n-1$.

Consider the element \(w:=\frac{\delta'+\gamma}{2n-2}\) and then we have the following equality in $Q(X, \ZZ)$.
\begin{align*}
    &w-\sigma_Q(w)\\
    =&\frac{e(\delta)+\gamma-\sigma_Q(e(\delta)+\gamma)}{2n-2}\\
        =&\frac{e(\delta)+\gamma-\sigma(e(\delta))-\gamma}{2n-2}\\
    =&\frac{e(\delta)+e(\sigma(\delta))}{2n-2}\\
    =&\frac{e(\delta_+)}{n-1}\\
    =&e(x)
\end{align*}
where the second equality uses that $\overline{c}_2(X)$ is preserved by $\sigma_Q$, the third equality uses
\eqref{eqn:eAndsigma-1}, and the fourth equality uses $\sigma(\delta)=\delta_+-\delta_-$. Since $x$ is not divisible by 2, neither is $w-\sigma(w)$. Therefore the Comessatti characteristic of the involution $\sigma_Q$ on $Q(X,\ZZ)$ is at least 1. By \Cref{lemma:ComessattiQuotient}, the Comessatti characteristic of the involution $\sigma$ on $H^4(X,\ZZ)$ is at least 1. This contradicts the maximality of $\sigma$.\\

 \textbf{For Case 4} of \Cref{lemma:LatticeClassification}, the proof is the same as in  Case 3. The only difference is that now $A_{L^{\sigma}}\cong \ZZ/(n-1)\ZZ$ and $A_{L^{\sigma-}}\cong \ZZ/{2\ZZ}$, so we have $\tau(\sigma)=1$. Again by \Cref{thm:Markman-Q}, we have the following (instead of \eqref{eqn:eAndsigma-1}):
  \begin{equation}
    \sigma(e(\alpha))=e(\sigma(\alpha)).    
\end{equation}
We make the same choices of $\delta\in L$ and $\delta'\in L'$. By the same argument as in Case 3, we see that $$\delta=2x+(n-1)y$$ with $x\in L^\sigma$ and $y\in L^{\sigma-}$ such that $x$ is not divisible by $n-1$ and $y$ is not divisible by 2. 

We consider again the element \(w:=\frac{\delta'+\gamma}{2n-2}\), then a similar computation as in Case 3 yields that 
\begin{equation}
    w-\sigma_Q(w)=e(y)
\end{equation}
which is not divisible by 2. Hence the Comessatti characteristic of $Q(X,\ZZ)$, and hence also that of $H^4(X,\ZZ)$, is at least 1. 
\end{proof}

\subsection{Anti-holomorphic involutions}
Using hyper-K\"ahler rotation, we can deduce results on real structures from results on anti-symplectic involutions:
\begin{proof}[Proof of \Cref{cor:main:NoMaxRealStructureHKnOdd}]
	Since maximality is a topological property, we can use hyper-K\"ahler rotation to change the complex structure. More precisely, let $X$ be a hyper-K\"ahler manifold of K3$^{[n]}$-type ($n\geq 2$), and let $\sigma$ be a real structure with respect to the original complex structure $I$. By \Cref{prop:HKRotation}, there is another complex structure $K$ on $X$ such that $\sigma$ is holomorphic and anti-symplectic. Applying \Cref{thm:main:NonMaxHKnOdd} to the anti-symplectic holomorphic involution $\sigma$ on $(X, K)$, we deduce that $\sigma$ is not maximal.    
\end{proof}

\section{Final comments and questions}\label{sec:final_comments}

\subsection{Generalization beyond biregular involutions}
By inspecting the proof of \Cref{thm:main:NonMaxHKnOdd}, one sees that what is actually proved is the following more general result.

\begin{theorem}
\label{thm:GeneralForm}
    Let $X$ be a compact hyper-K\"ahler manifold of $\K3^{[n]}$-type with $n\geq 2$. 
    Let $\sigma\in \Mon_{\Hdg}(X)$ be an order-2 monodromy operator of $X$ preserving Hodge structure on $H^2(X,\ZZ)$. If $\sigma$ is anti-symplectic, i.e.~acts on $H^{2,0}(X)$ by $-\id$,
    then $\sigma$ acts on $H^2(X,\F2)\oplus H^4(X, \F2)$ non-trivially. In other words, the Comessatti characteristic of the involution acting on $H^2(X,\ZZ)$ and $H^4(X, \ZZ)$ cannot be simultaneously zero.
\end{theorem}
\begin{proof}
    The only place in the proof of \Cref{thm:main:NonMaxHKnOdd} where we used that $\sigma$ is a geometric involution is at the beginning of the proof of \Cref{lemma:LatticeClassification}, where we argued that the signature of $H^2(X,\ZZ)^\sigma$ is $(\geq 1, -)$ since $\sigma$ preserves a K\"ahler class. But the claim for the signature holds in general for any Hodge monodromy operator: since $\sigma\in \Mon^2(X)$ is of trivial spinor norm, it preserves $\mathcal{C}^{\circ}_X$, the connected component of the positive cone containing K\"ahler classes. Since $\mathcal{C}^{\circ}_X$ is convex, for any K\"ahler class $\omega$, the class $\frac{\omega+\sigma^*(\omega)}{2}$ is $\sigma$-invariant with positive square. Hence the signature of $H^2(X,\ZZ)^\sigma$ is $(\geq 1, -)$. The rest of the proof of \Cref{lemma:LatticeClassification} goes through verbatim. 
\end{proof}

In addition to \Cref{thm:main:NonMaxHKnOdd} as a special case, \Cref{thm:GeneralForm} can be applied more generally to birational automorphisms, since
a birational automorphism on a compact hyper-K\"ahler manifold induces a Hodge monodromy operator on its cohomology. Consequently, \Cref{thm:GeneralForm} implies the following result:
\begin{cor}
\label{cor:Bir}
 Let $X$ be a compact hyper-K\"ahler manifold of $\K3^{[n]}$-type with $n\geq 2$. Then any anti-symplectic birational involution acts non-trivially on $H^2(X,\F2)\oplus H^4(X, \F2)$.
\end{cor}
It will be interesting to extract some nice geometric consequences (\textit{e.g.}~certain generalization of non-maximality) from the non-triviality of the induced involution on cohomology with $\F2$-coefficients.
\begin{rmk}
 Let us clarify the action of a birational automorphism $f$ on the cohomology of $X$: it is in general \textit{not} given by the self-correspondence induced by $\overline{\Gamma_f}$, the closure of the graph of $f$, but by the following procedure. By \cite[Theorem 2.5]{Huybrechts-KahlerCone}, there exist two families of smooth hyper-K\"ahler manifolds $\mathcal{X}_1, \mathcal{X}_2$  over a pointed smooth curve $(S, 0)$ both with fiber over $0$ isomorphic to $X$, and an $S$-birational isomorphisms $F \colon \mathcal{X}_1\dashrightarrow \mathcal{X}_2$, such that $F$ restricted to the fibers over $0$ is $f$, and over $S\backslash\{0\}$, $F$ is a biregular isomorphism. 
 The closure of the graph of the isomorphism $F$ is a cycle in $\mathcal{X}_1\times_{S} \mathcal{X}_2$, its specialization at $0\in S$ gives rise to a cycle in $X\times X$:
 $$\gamma:= \operatorname{sp}(\overline{\Gamma_{F}}).$$
 Then the action of $f$ on the cohomology $H^*(X)$ used in \Cref{cor:Bir} is given by the correspondence by $\gamma$.
\end{rmk}

\subsection{Hyper-K\"ahler varieties of other deformation types}
Our main results on hyper-K\"ahler manifolds of K3$^{[n]}$-type lead us to the following natural question.
\begin{question}
    In general, do there exist maximal brane involutions on compact hyper-K\"ahler manifolds of dimension $>2$? How about the other known deformation types \textit{e.g.} Kum\(_n\)-type, OG10-type and OG6-type?
\end{question}

It would be extremely interesting to exploit the general distinctive properties of compact hyper-Kähler manifolds in order to answer this question. One important extra structure on their cohomology is the Looijenga--Lunts--Verbitsky (LLV) Lie algebra action \cite{LooijengaLunts} \cite{Verbitsky-CohomologyofHK}. We did not explicitly use the LLV-action in this paper, but it plays a crucial role in the proof of many results that we used.

\begin{rmk}
    Several generalizations of hyper-K\"ahler manifolds to the singular setting exist in the literature. For example, the study of symplectic orbifolds was initiated by Fujiki \cite{Fujiki-SymplecticOrbifolds}, and the more general notion of symplectic varieties was introduced by Beauville \cite{Beauville-SymplecticSingularities}. In contrast to the non-existence results of maximal involutions obtained in this paper in the smooth setting, it is relatively easy to construct examples of maximal involution on singular symplectic varieties. Indeed, given a space with maximal involution, Franz \cite{Franz} showed that any symmetric power of the space equipped with the induced involution is again maximal. Hence for a K3 surface $S$ equipped with a maximal real structure or anti-symplectic involution, the symplectic orbifold $S^{(n)}$ is maximal.
\end{rmk}

\begin{rmk}[Other varieties with $c_1=0$]
    Recall that the Beauville--Bogomolov decomposition theorem \cite{Beauville} says that any compact K\"ahler manifold with vanishing first Chern class admits a finite \'etale cover that is isomorphic to a product of a complex torus with strict Calabi--Yau manifolds and hyper-K\"ahler manifolds. Maximal real abelian varieties exist in any dimension \cite{Comessatti-RealAVI} \cite{Comessatti-RealAVII}. It is expected that maximal real smooth Calabi--Yau hypersurfaces can be constructed by patch-working techniques in any dimension. For the moment complete proofs are available only up to dimension 3, see \cite{Demory} and the references therein.
\end{rmk}

\subsection{(BBB)-brane involutions on compact hyper-Kähler manifolds}
\label{subsec:BBB-not-optimal}
\Cref{thm:main:NonMaxHKnOdd} and \Cref{cor:main:NoMaxRealStructureHKnOdd} can be reinterpreted as the claim that \textit{the Smith inequality \Cref{thm:SmithThom} is \emph{not} optimal for anti-symplectic or anti-holomorphic involutions on compact hyper-K\"ahler manifolds of K3$^{[n]}$-type for $n\geq 2$.}

Of course, \Cref{thm:main:NonMaxHKnSymp} can be reformulated similarly, but we want to argue that the upper bound given by the Smith inequality for holomorphic symplectic involutions on compact hyper-K\"ahler manifolds is presumably \textit{very far} from optimal. Already for K3 surfaces, any symplectic involution has 8 isolated fixed points, whose total $\F2$-Betti number is 8, which is much smaller than the one imposed by the Smith inequality (24 in this case). 

In fact, we have more precise information on the fixed loci of symplectic involutions on compact hyper-K\"ahler manifolds of \textit{known} deformation types. Let $X$ be a compact hyper-K\"ahler manifold equipped with a holomorphic symplectic involution $\sigma$.
\begin{itemize}
    \item If $X$ is of K3$^{[n]}$-type, by \cite[Theorem 1.0.3]{Kam_Mon_Obl-SymplecticAutomorphisms}, the fixed loci of $\sigma$ is the disjoint union of hyper-K\"ahler manifolds deformation equivalent to Hilbert schemes of points on K3 surfaces\footnote{The precise number of connected components in each dimension is given in  \cite[Theorem 1.0.3]{Kam_Mon_Obl-SymplecticAutomorphisms}.}.
    A computation in the first few values of $n$ shows that the total Betti number of the fixed locus tends to be much smaller than the total Betti number of $X$.
    For example, when $n=2$, the fixed locus consists of a K3 surface and 28 points, whose  total Betti number is 52; when $n=3$, the fixed locus consists of 8 K3 surfaces and 64 points, whose total Betti number is 256; when \(n=4\), the fixed locus consists of 126 points, 28 K3 surfaces and a hyper-Kähler manifold of K3\(^{[2]}\)-type, whose total Betti number is 1122. The manifold $X$ has total Betti number 324 for \(n=2\), 3200 for \(n=3\) and 25650 for \(n=4\). 
   
    \item If $X$ is of $\operatorname{Kum}_n$-type, some partial results are obtained in \cite[Theorem 1.0.4]{Kam_Mon_Obl-SymplecticAutomorphisms}.
    \item If $X$ is of OG6-type, by \cite[Theorem 1.1]{GGO}, $\sigma$ acts trivially on the second cohomology. Such involutions are classified in \cite[Theorem 5.2]{Mongardi-Wandel-AutOG}, and their fixed loci are always disjoint unions of K3 surfaces and isolated points.
    \item If $X$ is of OG10-type, by \cite[Theorem 1.1]{GGOV}, there is no non-trivial finite-order symplectic automorphism.
\end{itemize}

\begin{question}
    Can we establish a general sharper upper bound for the total $\F2$-Betti number of the fixed locus of a holomorphic symplectic involution on a compact hyper-K\"ahler manifold, than the one provided by the Smith inequality?
\end{question}

We provide a proof of \Cref{thm:main:NonMaxHKnSymp} covering also OG6-type hyper-K\"ahler manifolds.
The key observation is the following elementary fact in topology:
\begin{lemma}
\label{lemma:b*=chi}
    Let $X$ be a compact manifold equipped with an involution $\tau$ acting non-trivially on $H^*(X,\QQ)$. Assume $H^{\operatorname{odd}}(X, \QQ)=0$ and  $H^{\operatorname{odd}}(X^\tau, \F2)=0$. Then $\tau$ is not maximal.
\end{lemma}
\begin{proof}
    Let $b_*(-, \F2)$ denote the total $\F2$-Betti number. 
    \begin{align*}
        &b_*(X^{\tau}, \F2)\\
        =&\chi(X^\tau)\\
        =&2 \chi(X/\tau)-\chi(X)\\
        =& \sum_i \left(2\dim H^i(X,\QQ)^\tau-\dim H^i(X,\QQ)\right)\\
        <&\sum_i\dim H^i(X,\QQ)\\
        \leq& b_*(X, \F2)
    \end{align*}
    where the first equality is by the hypothesis that $H^{\operatorname{odd}}(X^\tau, \F2)=0$, the third equality is by the vanishing of $H^{\operatorname{odd}}(X,\QQ)$, the strict inequality follows from the assumption that $\tau$ acts non-trivially on $H^*(X,\QQ)$, the last inequality is from the universal coefficient theorem. 
\end{proof}
\begin{cor}
\label{cor:NoMaxBBB}
    A non-trivial holomorphic symplectic involution on a compact hyper-K\"ahler manifold of K3$^{[n]}$-type or OG6-type is not maximal. 
\end{cor}
\begin{proof}
It suffices to check the conditions of \Cref{lemma:b*=chi}. By \cite[Proposition 10]{Beauville2} for K3$^{[n]}$-type, \cite[Theorem 5.2 and Remark 6.9]{Mongardi-Wandel-AutOG} for OG6-type, there is no nontrivial automorphisms acting trivially on $H^*(X, \QQ)$.  By \cite{Goettsche} for K3$^{[n]}$-type, \cite{MRS-OG6} for OG6-type,
$H^{\operatorname{odd}}(X,\QQ)=0$.  By \cite[Theorem 1.0.3]{Kam_Mon_Obl-SymplecticAutomorphisms} for K3$^{[n]}$-type, \cite{GGO} and \cite{Mongardi-Wandel-AutOG} for OG6-type,
the fixed locus $X^\tau$ is a disjoint union of hyper-K\"ahler manifolds of K3$^{[m]}$-type (for various $m$), hence $H^{\operatorname{odd}}(X^{\tau}, \F2)=0$.
\end{proof}


    

\subsection{Galois-maximality}
 The non-existence of maximal brane involutions on manifolds of K3\(^{[n]}\) type is obtained by proving that the action of the involution is not trivial on \(H^*(X,\mathbb{F}_2)\), which is part of condition 3 of \Cref{prop:MaximalityViaCohomology}.
 Another way to obtain the non-existence of maximal brane involutions concerns the behavior of the Leray--Serre spectral sequence associated with the involution. Involutions for which the Leray--Serre spectral sequence degenerates are called \textit{Galois-maximal}, which is a weaker version of the maximality.
 
 It is then natural to investigate the degeneration of the Leray--Serre spectral sequence for compact hyper-K\"ahler manifolds with real structures. 

\begin{question}
   Let $X$ be a compact hyper-K\"ahler manifolds equipped with a real structure. If $X(\RR)\neq\emptyset$, when does the spectral sequence 
		\begin{equation}
			E_2^{p,q}=H^p(G, H^q(X, \F2))\Rightarrow H^{p+q}_G(X, \F2)
		\end{equation}
		degenerates at $E_2$ ? 
\end{question}


\newcommand{\etalchar}[1]{$^{#1}$}
\bibliographystyle{abbrv}
\bibliography{main}

\medskip \medskip

\noindent Simone Billi\\
\noindent{Fakultät für Mathematik und Informatik, Universität des Saarlandes, Campus E2.4, 66123
Saarbrücken, Germany}\\
\medskip \noindent{\texttt{simone.billi.96@gmail.com}}

\noindent Lie Fu\\
\noindent{Universit\'e de Strasbourg, Institut de recherche mathématique avancée (IRMA),  France} \\
\medskip \noindent{\texttt{lie.fu@math.unistra.fr}}

\noindent Annalisa Grossi\\
\noindent {Alma Mater Studiorum - Università di Bologna, Dipartimento di Matematica, Piazza 
	di porta San Donato 5, 40126 Bologna}\\
\medskip \noindent{\texttt{annalisa.grossi3@unibo.it}}

\noindent{Viatcheslav Kharlamov}\\
\noindent{Universit\'e de Strasbourg, Institut de recherche mathématique avancée (IRMA), France}\\
\medskip \noindent{\texttt{kharlam@unistra.fr}}

\end{document}